\documentclass{amsart}

\usepackage{latexsym}
\usepackage{amssymb}
\usepackage{epsfig}
\usepackage{color}
\usepackage{multirow}
\usepackage{wrapfig}

\newtheorem{thm}{Theorem}[section]

\newtheorem{lem}[thm]{Lemma}

\theoremstyle{definition}

\theoremstyle{remark}
\newtheorem{rem}{Remark}[section]

\newcommand{\R}{\mathbb{R}}

\makeatletter
\def\serieslogo@{}
\makeatother \makeatletter
\def\@setcopyright{}
\makeatother

\def\Real{{\mathbb R}}

\def\Cnumbers{{\mathbb C}}
\def\supp{{\rm supp}}
\newcommand{\eq}[1]{{(\ref{#1})}}
\newcommand{\Vector}[1]{{\left(\begin{matrix} #1 \end{matrix}\right)}}

\newcommand{\bv}[1]{{\mbox{\boldmath $ #1$}}}

\def\x{{\bv x}}
\def\y{{\bv y}}

\def\xsm{{\bv{\scriptstyle x}}}

\def\psm{{\bv{\scriptstyle p}}}

\def\p{{\bv p}}
\def\e{{\bv e}}

\def\Ai{{\rm Ai}}
\def\Bi{{\rm Bi}}
\def\Scal{{\mathcal S}}
\def\Fcal{{\mathcal F}}

 \newcommand{\ba}{\begin{array}}
 \newcommand{\ea}{\end{array}}
 \newcommand{\be}{\begin{equation}}
 \newcommand{\ee}{\end{equation}}

 \def\CC{\rm \hbox{C\kern-.57em\raise.47ex
 \hbox{$\scriptscriptstyle |$}\kern+0.5 em}}

\begin{document}
\title[]{Error estimates for Gaussian beams at a fold caustic}
\author{Olivier Lafitte}
\address{LAGA, UMR7539 
Universit\'e Sorbonne Paris Nord,
99, Avenue J. B. Clement, 93430
Villetaneuse, France}
\address{IRL CNRS-CRM, Universit\'e de Montr\'eal, Chemin de la tour, Montr\'eal, Canada.}
\email{lafitte@math.univ-paris13.fr}
\author{Olof Runborg$^*$}
\address{Department of Mathematics, KTH, 10044 Stockholm, Sweden}
\email{olofr@kth.se}
\thanks{$^*$ Corresponding author.}
\date{\today}
\maketitle

\begin{abstract}
In this work we show an error estimate for a first order
Gaussian beam 
at a fold caustic, approximating time-harmonic
waves governed by the Helmholtz equation.
For the caustic that we study the
exact solution can be constructed using Airy functions and
there are explicit formulae for the Gaussian beam parameters.
Via precise comparisons we show that the pointwise error on the
caustic is of the order $O(k^{-5/6})$ where $k$ is the
wave number in Helmholtz.
\end{abstract}

\section{Introduction}

Gaussian beam superpositions is a high frequency asymptotic approximation 
for solutions of wave equations
\cite{Ralston:82}. It is used in 
numerical methods to simulate waves in the
high frequency regime.
Unlike standard geometrical optics, the Gaussian beam approximation does
not break down at caustics, which is one of its main advantages.

In this paper we consider error estimates for the approximation
in terms of the wave number $k>0$.
Error estimates for Gaussian beams are known in a number of settings.
See for instance \cite{Liu1,Liu2} and the references therein.
The main result is that, in $L^2$ and 
Sobolev norms, the relative error of first 
order beams decays as $O(k^{-1/2})$, independently
of dimension and regardless of the presence of caustics.
This has been shown for general strictly hyperbolic partial differential equations
and the Schr\"odinger equation \cite{Liu1,Liu2}
as well as the Helmholtz equation \cite{Liu3}.
The better
rate $O(k^{-1})$ is typically observed in numerical
computations and has been shown in $L^2$ for the Schr\"odinger
equation \cite{Zheng:14}, and also in $L^\infty$ for 
the Schr\"odinger and the acoustic wave equation
on sets strictly
away from caustics
\cite{Liu2}. Similar estimates have also been derived
for higher order beams.
For $p$-th order beams the rates are $O(k^{-p/2})$ and
$O(k^{-\lceil p/2\rceil})$ respectively.
There are, however, no precise, pointwise, error estimates for the solution
at a caustic. In particular, for first order beams it has not been shown
that this error vanishes as $k\to\infty$, although there is
ample numerical evidence to this effect; see for instance \cite{LiuRalston:20}.

The purpose of this
paper is to show such an error estimate for a
typical fold caustic in two dimensions.
More precisely, we consider 
the Helmholtz equation
\be\label{eq:HHmain}
  \Delta u + k^2(1-x)u = 0.
\ee
We assume there is an incident wave $u_{\rm inc}$ from $x=-\infty$
making an angle $\Theta\in(0,\pi/2)$ with the $x$-axis.
Moreover, at $x=0$
it has the amplitude envelope $A(y)$, so that
$$
   u_{\rm inc}(0,y)= A(y)e^{ik\sin\Theta}.
$$
This wave will
generate a fold caustic at the line $x=x_c$ where
$$
  x_c=\cos^2(\Theta).
$$ 
Figure~\ref{fig:foldcaustic} shows
a representative solution.
In Section~\ref{sec:Exactsols}
we make this physical situation
precise and derive an exact solution using
Airy functions on $\Real^2$.
We subsequently study the solution
in the region $0\leq x\leq x_c$
and compare it at $x=x_c$
to an approximation
using Gaussian beams, denoted by $u_{GB}$.
(Note that $0< 1-x\leq 1$ in this region; 
we do not make comparisons elsewhere, as the
equation then no longer models the physical situation.)
The main result is the following theorem.
\newpage

\begin{figure}[t]
 \begin{center}
 \includegraphics[width=.9\textwidth]{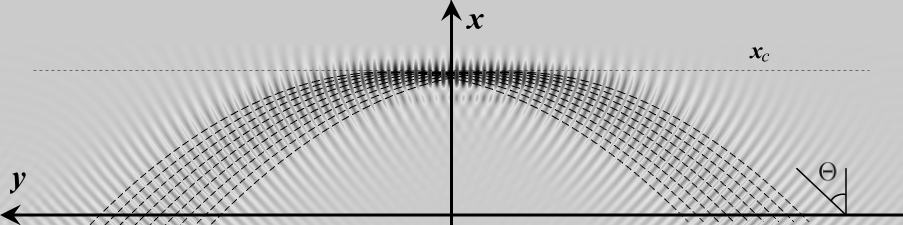}
 \end{center}
 \caption{The fold caustic: Example of solution with ray tracing picture
 for $0\leq x \leq 1$.
 }
 \label{fig:foldcaustic}
\end{figure}
\begin{thm}\label{mainresult}
Suppose $A$
is a Schwartz class function,
$A\in \Scal(\Real)$,
and that $0<\bar\Theta_0 \leq \Theta \leq \bar\Theta_1 <\pi/2$.
 Let $u$ be the exact solution
of \eqref{eq:HHmain} as
defined in Section~\ref{sec:Exactsols}
and 
 $u_{GB}$ a first order Gaussian beam approximation detailed in
 Section~\ref{sec:GBapp}.
Then there is a constant $C$ independent of $k$ such that
$$
  ||u_{GB}(x_c,\cdot)-u(x_c,\cdot)||_{L^\infty}
  \leq
   C 
     k^{-5/6}.
$$
\end{thm}
This result hence confirms that first order Gaussian beams 
do converge pointwise at the caustic. Moreover, since the
solution itself grows as $O(k^{1/6})$ at this caustic \cite{Ludwig}, the
relative error is $O(k^{-1})$, the same  as away from the caustic.
We conjecture that this will be the case also for more
general caustics.

The paper is organized as follows. In Section 2 notations are established
and some preliminary results are discussed. In Section 3 the exact solution
is defined and a formula for it is derived.
In Section 4, the corresponding Gaussian beam approximation is introduced.
Sections 5, 6 and 7 contain estimates of the Gaussian beam parameters, the phase,
various oscillatory integrals, as well as the exact solution and the Gaussian beam
approximation.
In Section 8 the proof of the main result in Theorem~\ref{mainresult} is
carried out. Finally, in Section 9 some properties of Airy functions are presented.

\section{Preliminaries}

In the analysis we use a $k$-scaled Fourier transform
and indicate it with a hat mark on the function,
\be\label{Fdef}
   \hat{f}(\eta) :=\Fcal_k(f)(\eta)
:=
   \sqrt{k}\Fcal(f)(k\eta) =
   \sqrt{\frac{k}{2\pi}}\int f(x) e^{-ikx\eta}dx.
\ee
The corresponding scaled inversion formula reads
$$
   f(x) = \sqrt{\frac{k}{2\pi}}\int \hat{f}(\eta) e^{ikx\eta}d\eta
   =:\Fcal^{-1}_k(\hat{f})(x)
   =\sqrt{k}\Fcal^{-1}(\hat{f})(kx).
$$
We also have
$$
\widehat{f_x}(\eta) = ik\eta \hat{f}(\eta),\qquad
   ||f||_{L^\infty}\leq \sqrt{\frac{k}{2\pi}}||\hat{f}||_{L^1},
   \qquad 
   \Fcal_k(f\ast g)(\eta)
   = \sqrt{\frac{2\pi}{k}}\hat{f}(\eta) \hat{g}(\eta)
$$

We will frequently make use of a smooth, even, cut-off function
which we denote
$\psi\in C_c^\infty(\Real)$. It is defined as
$$
 \psi(x) = \begin{cases}
 1, & |x|\leq 1, \\
 0, & |x|\geq  2, \\
 \in(0,1), &1<|x|<  2,
 \end{cases}
 \qquad \psi(-x)=\psi(x).
$$
This is used to divide integrals into subdomains and to regularize
the Fourier transform of functions in ${\mathcal S}'$, the space
of tempered distributions. For example, if 
$f$ is in $L^{\infty}(\Real)$,
but not in $L^1(\Real)$,
the definition \eq{Fdef} must be interpreted in distributional sense.
We then let $\psi_t:=\psi(x/t)$
and consider instead the 
Fourier transform of the compactly supported function $f\psi_t$,
which is well-defined by \eq{Fdef} for all $t>0$. The following Lemma shows that
the limit as $t\to\infty$ gives us the Fourier
transform in ${\mathcal S}'$.
\newpage

\begin{lem}\label{lem:Fregularize}
Let $f\in L^{\infty}$ and set $\psi_t=\psi((b_0x+b_1)/t)$ for any fixed
real numbers
$b_0\neq 0$ and $b_1$.
Then, with $\Fcal_k$ as defined above,
$$
\lim_{t\to \infty}\Fcal_k(f\psi_t) = \Fcal_k(f),
$$
in ${\mathcal S}'$.
Moreover, if 
$g\in{\mathcal S}$ then, again in ${\mathcal S}'$,
$$
\lim_{t\to \infty}\Fcal_k((f\ast g)\psi_t) =
\lim_{t\to \infty}\sqrt{\frac{2\pi}{k}}\Fcal_k(f\psi_t)\Fcal_k(g)=
\sqrt{\frac{2\pi}{k}}\Fcal_k(f)\Fcal_k(g).
$$
\end{lem}
The short proof is found in the Appendix.
In particular,
if $\hat{f}$ is defined pointwise, the Lemma shows that
\be\label{Fdefdist}
   \hat{f}(\eta) =\lim_{t\to\infty}
   \sqrt{\frac{k}{2\pi}}\int \psi(x/t) f(x) e^{-ikx\eta}dx.
\ee

We also introduce some notation that will prove useful later on
in the paper. We let 
\be\label{etaxidef}
   \xi_0 = \cos\Theta, \qquad
   \eta_0 = \sin\Theta,
\ee
so that $(\xi_0,\eta_0)^T$ is the unit vector pointing in the
propagation direction of the incident wave. 
Following Theorem~\ref{mainresult}
we will assume, throughout the paper, that
$0<\bar\Theta_0 \leq \Theta \leq \bar\Theta_1 <\pi/2$.
This translates to bounds
on $\xi_0$ and $\eta_0$ of the form 
\begin{equation}\label{xi0deltabound}
0<\bar\xi_0\leq \xi_0\leq \bar\xi_1<1, \qquad
0<\bar\eta_0\leq \eta_0\leq \bar\eta_1<1.
\end{equation}
for some $\bar\xi_j$ and $\bar\eta_j$.
Moreover, we let 
\be\label{deltadef}
  \delta = x_c-x = \xi_0^2-x,
\ee
be the distance to the caustic. 
Finally, we introduce 
the polynomial $q$, which is
related to the geometrical spreading of the rays,
\be\label{qdef}
   q(s) = 1+2is-s^2\beta, \qquad \beta = 1+2i\xi_0.
\ee
It will be used frequently in the analysis.


\section{Expression of the exact Helmholtz equation solution}
\label{sec:Exactsols}


In this section we define 
an exact solution to the Helmholtz
equation for the physical setup described in the introduction. 
Using a property
of the Airy function we deduce a
decomposition of the solution into
forward and backward going waves.

We consider a solution $u$ to \eq{eq:HHmain}
which is a tempered distribution on $\Real^2$, i.e.
$u\in {\mathcal S}'(\R^2)$.
The solution then has a $k$-scaled Fourier transform in $y$
which we denote $\hat{u}(x,\eta)$. 
Upon Fourier transforming also 
\eq{eq:HHmain} in $y$, we obtain an ODE for $\hat{u}(x,\eta)$,
\begin{equation}\label{uhatode2}
  \hat{u}_{xx} +k^2(1-x-\eta^2)\hat{u}= 0.
\end{equation}
The only tempered distribution solution to this ODE is given by
$$
  \hat{u}(x,\eta) = a(\eta)\Ai(k^{2/3}(x+\eta^2-1)),
$$
where $\Ai$ is the Airy function of the first kind, and
$a(\eta)$ is a function to be determined.
This solution is thus a $C^\infty$ bounded solution.
The Airy function in this expression contains waves going
both forward and backward. 
In the sequel, we will choose the function $a(\eta)$
as $k^{1/6}P(k,\eta)$ defined in \eq{vhatdef2}.
To arrive at this choice, we first note
that 
when 
$$
\alpha = \exp(i\pi/3),
$$ 
it holds for all $z$ that
\cite[Eq. 9.2.14]{NIST:DLMF},
\begin{equation}\label{airydecomp1}
\Ai(z) = \alpha \Ai\left(-\alpha z\right)
+\bar\alpha\Ai\left(-\bar\alpha z\right).
\end{equation}
This follows since $\Ai(-\alpha z)$ and
$\Ai(-\bar\alpha z)$ solve the same ODE
as $\Ai(z)$ given that $\alpha^3=-1$.
We then introduce the scaled variables
$$
\zeta(x)=k^{\frac23}(x+\eta^2-1),\qquad
\zeta_+(x)=-\alpha\zeta, \qquad
\zeta_-(x)=-\bar\alpha \zeta,
$$ 
such that
$$
\hat{u}(x,\eta) = a(\eta)\Ai(\zeta(x))
=a(\eta)\alpha\Ai(\zeta_+(x))+a(\eta)\bar\alpha\Ai(\zeta_-(x)).
$$
To further understand this decomposition, we note 
that the asymptotics of the Airy function (in the angular sector 
$|$arg$(z)|<2\pi/3$) is
$$
\Ai(z)\simeq \frac1{2\sqrt{\pi}}z^{-\frac14}\exp\left(-\frac23 z^{\frac32}\right).
$$
Therefore, upon defining the phase $\phi(x)=\frac23(1-\eta^2-x)^{2/3}$
we have, when $x<1-\eta^2$,
$$
 \Ai(\zeta_+(x))
\simeq Ck^{-1/6}\frac{e^{-ik\phi(x)}}
{(1-\eta^2-x)^{\frac14}},\qquad
 \Ai(\zeta_-(x))
\simeq Ck^{-1/6}\frac{e^{+ik\phi(x)}}
{(1-\eta^2-x)^{\frac14}}.
$$
Thus, the phases of the two expressions, and their gradients,  
have
opposite signs, 
meaning that the
two terms in \eq{airydecomp1}
represent a decomposition into forward and backward going waves
in the region $x<1-\eta^2$.
More precisely, the solution is fully known when $a(\eta)$ is given, and one assumes that this solution of 
$\Delta u +k^2(1-x)u=0$ for all $x$ (including $x<0$ where the velocity grows) is the sum of an incoming and an outgoing wave of the form 
$a(\eta) \alpha \Ai(\zeta_+(x))$ and $a(\eta)\bar\alpha\Ai\left(\zeta_-(x)\right)$, 
respectively.
This leads us to define 
\begin{align*}
 T_0(x,\eta) = \frac{\Ai\left(\zeta(x)\right)}{
  \Ai\left(\zeta(0)\right)},\qquad
 T_+(x,\eta) = \frac{\Ai\left(\zeta_+(x)\right)}{
  \Ai\left(\zeta_+(0)\right)},\qquad
 T_-(x,\eta) = \frac{\Ai\left(\zeta_-(x)\right)}{
  \Ai\left(\zeta_-(0)\right)}.
\end{align*}
which are three particular solutions of \eqref{uhatode2},
since each term in
\eqref{airydecomp1} solve the ODE.
These solutions are normalized such that they equal one for $x=0$.
The solutions $T_+$ and $T_-$ represent forward and backward
going waves.
Among the three solutions,
only $T_0$ is bounded,
since $\Ai(z)$ is bounded, while $T_\pm$ are not
even in ${\mathcal S}'$,
since $\Ai(\alpha z)$ includes
also the unbounded second kind Airy function $\Bi(z)$, cf. \eq
{eq:ABident}.

We are looking for the solution
$$
  \hat{u}(x,\eta) = T_0(x,\eta) \hat{u}(0,\eta).
$$
We do not know $\hat{u}(0,\eta)$, just that the incoming
part of $\hat{u}(0,\eta)$ represents the incident plane
wave.
We therefore write it as a sum
of an incoming and an outgoing part
$$
  \hat{u}(0,\eta)=\hat{u}_+(0,\eta)+\hat{u}_-(0,\eta),
$$
where we define $\hat{u}_+(0,\eta)$ as the $k$-scaled
Fourier transform in $y$ of the incoming wave
with amplitude ${A}$ and direction $\Theta$
(recall $\eta_0 = \sin\Theta$),
$$
 u_+(0,y) = u_{\rm inc}(0,y) = 
 A(y)e^{ik\eta_0 y}.
$$
We then
want to find $\hat{u}_-(0,\eta)$ such that
$$
T_0(x,\eta) \hat{u}(0,\eta) = T_+(x,\eta) \hat{u}_+(0,\eta)+T_-(x,\eta) \hat{u}_-(0,\eta).
$$
To achieve this, it is necessary and sufficient that the
values of the functions and their derivatives agree at $x=0$, 
since both sides satisfy the same second order ODE. This gives us the
linear relations
\begin{align*}
  \hat{u}(0,\eta)&=\hat{u}_+(0,\eta)+\hat{u}_-(0,\eta),\\
T'_0(0,\eta) \hat{u}(0,\eta) &= T'_+(0,\eta) \hat{u}_+(0,\eta)+T'_-(0,\eta) \hat{u}_-(0,\eta),
\end{align*}
from which we can deduce
$$
\hat{u}_-(0,\eta) = \frac{T'_+(0,\eta) -T'_0(0,\eta)}{T'_0(0,\eta) -T'_-(0,\eta)}
\hat{u}_+(0,\eta)=: T(\eta)\hat{u}_+(0,\eta).
$$
It follows that
$$
\hat{u}(x,\eta) = T_0(x,\eta) \hat{u}(0,\eta)
= T_0(x,\eta) (\hat{u}_+(0,\eta)+\hat{u}_-(0,\eta))
= T_0(x,\eta)(1+T(\eta))\hat{u}_+(0,\eta).
$$
The decomposition is not valid at the roots of $\Ai(\zeta(0))=\Ai(k^{\frac23}(\eta^2-1))$. Another form is available, which is valid at all points.
It is given in the following Lemma.

\begin{lem}
One can express $T(\eta)$ as follows
$$T(\eta)=-\alpha\frac{\Ai(\zeta_-(0))}{\Ai(\zeta_+(0))},
\qquad
1+T(\eta) = 
\bar\alpha\frac{\Ai(\zeta(0))}{\Ai(\zeta_+(0))},
$$
\end{lem}
\begin{proof}
To simplify notation we write $\Ai_{+,-,0}(x):=\Ai(\zeta_{+,-,0}(x))$.
Then, using \eq{airydecomp1},
\begin{align*}
T_\pm-T_0 &= 
\frac{\Ai_\pm(x)}{\Ai_\pm(0)}-\frac{\Ai_0(x)}{\Ai_0(0)}
=\frac{\Ai_\pm(x)\Ai_0(0)-\Ai_\pm(0)\Ai_0(x)}{\Ai_\pm(0)\Ai_0(x)}\\
&=\frac{\Ai_\pm(x)
(\alpha\Ai_+(0)+\bar\alpha\Ai_-(0))
-\Ai_\pm(0)(\alpha\Ai_+(x)+\bar\alpha\Ai_-(x))}{\Ai_\pm(0)\Ai_0(x)}
\\
&=\pm\alpha^{\mp 1}\frac{\Ai_-(0)\Ai_+(x)-\Ai_+(0)\Ai_-(x)}{\Ai_\pm(0)\Ai_0(x)}.
\end{align*}
Hence,
$$
T_+-T_0 = 
(T_--T_0)\frac{\Ai_-(0)}{\Ai_+(0)}(-\bar\alpha^2)=
(T_--T_0)\frac{\Ai_-(0)}{\Ai_+(0)}\alpha.
$$
The first identity follows upon differentating this expression with respect to $x$.
The second identity is then given by another application of
\eq{airydecomp1}.
\end{proof}

It follows now that
$$
  T_0(x,\eta)(1+T(\eta)) = \bar\alpha\frac{\Ai(\zeta(x))}{\Ai(\zeta(0)}
  \frac{\Ai(\zeta(0))}{\Ai(\zeta_+(0))}
  = 
  \bar\alpha\frac{\Ai(\zeta(x))}{\Ai(\zeta_+(0))}.
$$
Since we know $\hat{u}_+(0,\eta)$ we can thus express the full solution 
as 
$$
  \hat{u}(x,\eta_0+\eta) = \bar\alpha\frac{\Ai(k^{\frac23}(x-X))}{\Ai(\alpha k^{\frac23}X)}\hat{u}_+(0,\eta_0+\eta),
$$
where we defined
$$
   X(\eta) = 1- (\eta_0+\eta)^2.
$$
In this expression, one notices that the denominator never vanish because all the roots of the Airy function are on the negative real axis.

Finally, since 
$$
 u_+(0,y) = A(y)e^{ik\eta_0 y} \quad\Rightarrow\qquad
  \hat{u}_+(0,\eta) =\hat{A}(\eta-\eta_0),
$$
we get
$$
  \hat{u}(x,\eta_0+\eta) = \bar\alpha\frac{\Ai(k^{\frac23}(x-X))}{\Ai(\alpha k^{\frac23}X)}
  \hat{A}(\eta).
$$
We write this as
$$
{ \hat u}(x, \eta_0+\eta)={ \hat v}(\eta,x,k)
 {\hat{A}}(\eta),
$$
where
\begin{equation}\label{vhatdef2}
{\hat v}(\eta,x,k) = k^{1/6} P(k,\eta)\Ai(k^{\frac23}(x-X)),
\qquad P(k,\eta) = 
\frac{\bar\alpha k^{-1/6}}{\Ai(\alpha k^{2/3}X)}.
\end{equation}


\section{Construnction of the Gaussian beam approximation}
\label{sec:GBapp}

In this section we derive expressions for a
first order Gaussian beam approximation 
to the solution of \eqref{eq:HHmain}.
A Gaussian beam is a high frequency asymptotic solution to
the Helmholtz equation. To model a general solution of
\eqref{eq:HHmain}, superpositions of Gaussian beams are used.
We give the general form of a Gaussian beam and their superposition
 in $\Real^2$ below. The derivations of the expressions
 can be found in \cite{Liu3}.
 
The Helmholtz equation with a general index of refraction $n(x)$
reads
\be\label{eq:HHgeneral}
  \Delta u + k^2n(x)^2u = 0.
\ee
When $n(x)$ is real, the equation models wave propagation, but
it has a well-defined solution also when $n(x)$ is imaginary.
However, Gaussian beams can only be defined for real $n(x)$.

A first order Gaussian beam for \eq{eq:HHgeneral} has the form
\begin{equation} \label{eq:gbform}
  v_b(\x) = 
  a(s)e^{ik(S(s)+(\xsm-\gamma(s))\cdot\psm(s)
  +\frac12(\xsm-\gamma(s))^TM(s)(\xsm-\gamma(s))},\qquad s=s^*(\x),
\end{equation}
where $a(s)\in\Cnumbers$ is the amplitude, $S(s)\in \Real$ the reference phase, 
$\p(s)\in\Real^2$ the phase gradient 
and $M(s)\in \Cnumbers^{2\times 2}$ the phase Hessian.
Moreover, $\gamma(s)\in\Real^2$ is the {\it central ray},
which agrees with the rays of geometrical optics.
An example of a Gaussian beam is shown in Figure~\ref{fig:gbsketch}.
In \eqref{eq:gbform}
the parameter $s$ depends on the point of evaluation $\x$.
Normally one takes the $s$-value for
the point on the central ray that is closest to $\x$, as indicated in 
Figure~\ref{fig:gbsketch}. However, in the analysis below we
make a simpler choice.
By a result in \cite{Ralston:82}
the Hessian $M$ will always have a positive definite imaginary part.
The solution $v_b$ will therefore be a "fattened" version of the central ray,
with a Gaussian profile normal to the ray with a width determined by
$M$. 
\begin{figure*}[t]
 \begin{center}
 \includegraphics[width=.3\linewidth]{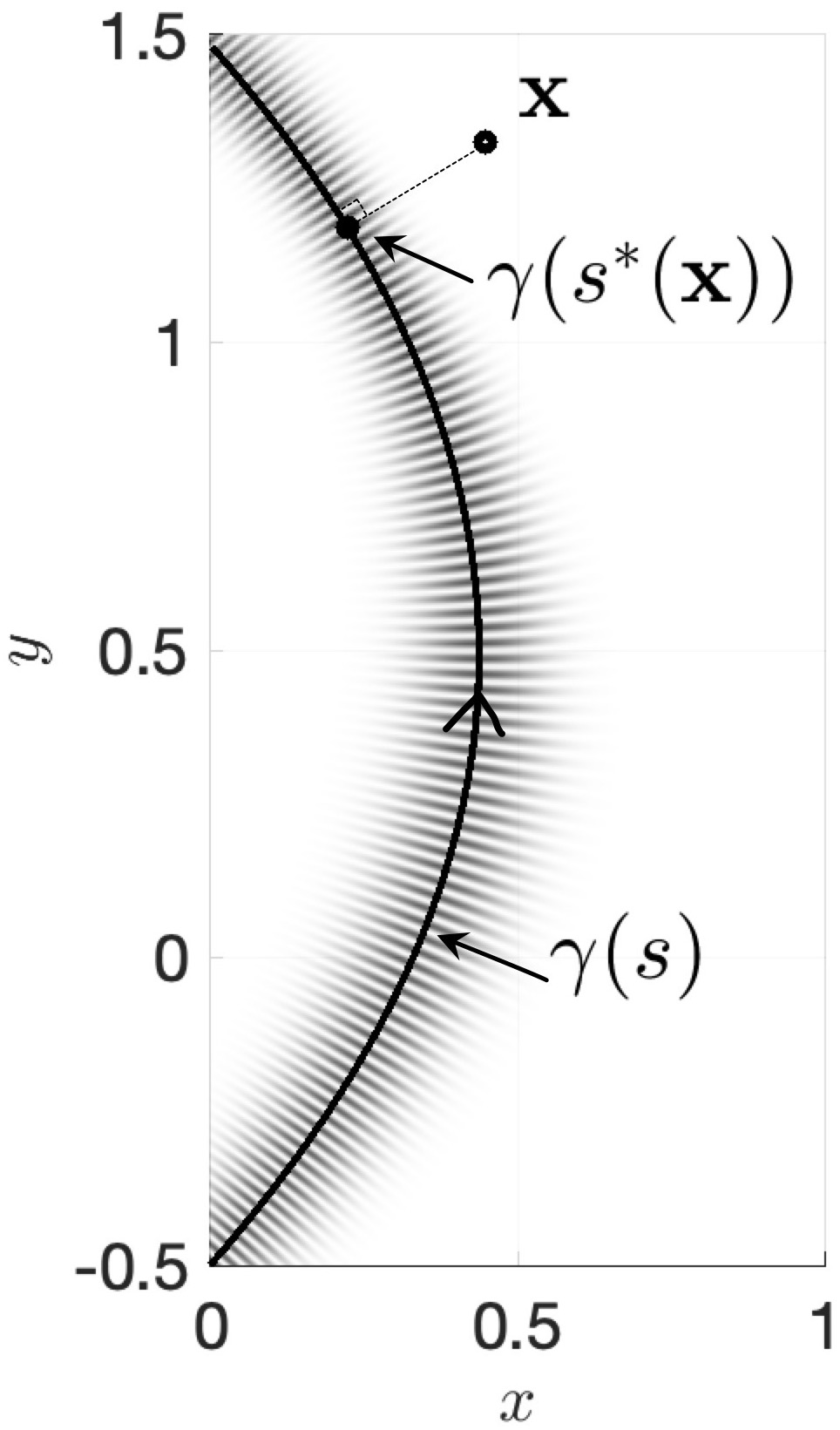}
 \end{center}
 \caption{A Gaussian beam starting at $(x,y)=(0,-0.5)$ with
 direction $\eta_0=\sin\Theta=3/4$. The central ray $\gamma$ is indicated with
 a solid line.
 }
 \label{fig:gbsketch}
\end{figure*}

The $s$-dependent parameters in the Gaussian beam are all given by
ODEs \cite{Liu3}, as follows
\begin{align}\label{eq:GBode}
\frac{d\gamma}{ds}&= 2\p,\qquad
 \frac{d\p}{ds} =\nabla n^2(\gamma),\qquad
 \frac{dM}{ds} =D^2(\nabla n^2)(\gamma)-2M^2,\\
 \frac{dS}{ds} &=2n^2(\gamma),\qquad
 \frac{da}{ds} =-{\rm tr}(M)a.\nonumber
\end{align}
The initial data for $\gamma$ and $\p$ is given by the starting point ($x_0,y_0)$
and direction $(\xi_0,\eta_0)$
of the beam,
$$
  \gamma(0)=\Vector{x_0 \\y_0}, \qquad \p(0)=\Vector{\xi_0 \\ \eta_0}.
$$
In order to form an admissible Gaussian beam, $M(0)$
must always satisfy
\begin{equation}\label{eq:M0cond}
M(0)^T=M(0), \qquad
M(0)\gamma'(0)=\p'(0),
 \qquad {\bv a}^T({\rm Im} M(0)){\bv a}>0, \quad 
 \text{when ${\bv a}\perp\gamma'(0)$.}
\end{equation}
The choice of $M(0)$ and the precise form of the incoming wave 
finally determine the initial data for $S$ and $a$. 
We come back to this issue below.

To build more general solutions we use superpositions of Gaussian beams.
We assume that the incoming wave is known along a curve $\Gamma$ in $\Real^2$,
which we can parameterize with the parameter $z$, so that $\Gamma(z) = (x_0(z),y_0(z))$.
For each point on $\Gamma$ we launch one Gaussian beam in the direction
$\theta_0(z)$
of the wave at that point. The parameters of the
beams then also depend on $z$ and we write $\gamma=\gamma(s;z)$, $\p=\p(s;z)$, etc.
This gives the beams $v_b=v_{\rm beam}(\x;z)$, from which we finally construct
the Gaussian beam superposition
\begin{equation}\label{eq:gbsuperpos}
u_{GB}(\x) = \sqrt{\frac{k}{2\pi}}\int v_{\rm beam}(\x;z)dz.
\end{equation}
See \cite{Liu3} for more details. In a numerical scheme the
$z$-variable is discretized and for each discrete value,
the ODEs \eqref{eq:GBode} are solved with a numerical ODE
method. The superposition \eqref{eq:gbsuperpos}
is subsequently computed using numerical quadrature.

\subsection{Expressions for Gaussian beam parameters}

In \eq{eq:HHmain} we have the index of refraction
$n(x)=\sqrt{1-x}$.
 The ODEs
\eqref{eq:GBode}
 can be solved explicitly and we
get analytic formulae for all parameters in the Gaussian
beam.
To show this we
 let $\p=(\xi,\eta)$ and $\gamma=(x,y)$ and also 
 recall that
$
\p(0)=
({\xi_0,\eta_0}),
$
where 
$
\xi_0^2+\eta_0^2=1.
$
Since we let the beam start at $x_0=0$, we set
 $\gamma(0)=(0,y_0)$.
With the particular choice of $n$ 
it follows from \eqref{eq:GBode} that
$\xi'(s) = -1$ and $\eta'(s)=0$.
Hence,
$$
 \xi(s) = \xi_0-s,\qquad
 \eta(s) = \eta_0.
$$
For the positions, we get $x'(s)=2\xi(s)=2\xi_0-2s$ and
$y'(s)=2\eta(s)=2\eta_0$.
By also using the initial data we obtain
$$
 x(s) =2s\xi_0-s^2,\qquad 
 y(s) =y_0+2s\eta_0.
$$
The caustic $x=x_c$ is located at the point where the ray
turns back, i.e. where $x'(s)=0$, which gives $s=\xi_0$ and
$$
  x_c=x(\xi_0)=\xi_0^2.
$$ 
Note that all the rays are confined to the region $x\leq x_c<1$,
where the index of refraction is real-valued. The fact that
$n(x)$ is complex-valued for $x>1$ therefore
does not affect the Gaussian beams.

We also need to compute the coefficients corresponding to the
phase $S$, the second derivative of the phase $M$ and the 
amplitude $a$. We have
\begin{align*}
\frac{dS}{ds} &=2 n^2(x(s)) = 2(1-x(s)) = 2[1-2s\xi_0+s^2],\qquad S(0)=S_0,
\end{align*}
so the phase is a third order polynomial,
$$
S(s) = S_0+2s-2s^2\xi_0+\frac23s^3.
$$
For $M$ we have the Riccati equation
\begin{align*}
 \frac{dM}{ds} &=D^2(n^2)(x(s))-2M(s)^2 = -2M(s)^2, \qquad M(0)=M_0,
\end{align*}
with the solution
$$
  M(s)=(I+2sM_0)^{-1}M_0.
$$
The matrix $M_0$ must satisfy the conditions in \eqref{eq:M0cond}.
We pick 
$$
M_0=Q+iP, \qquad
  P = \Vector{\eta_0^2 & -\eta_0\xi_0\\
  -\eta_0\xi_0 & \xi_0^2},\qquad
  Q = \frac12 \Vector{-\xi_0 & -\eta_0\\
-\eta_0 & \xi_0}.
$$
Note that $P,Q$ are symmetric,
$P$ is the orthogonal projection on $\p_0^\perp$,
and 
$2Q\p_0=-\e_1=\p'(0)$. Moreover, $(\p_0^\perp)^TQ\p_0^\perp = \xi_0>0$.

Next, one checks that
$$
M(s) = (I+2sM_0)^{-1}M_0 = \frac{1}{q(s)}(I+2s(iI-M_0))M_0,
$$
where 
$$
\det(I+2sM_0)= 1+2s\mbox{Tr}M_0+(2s)^2\mbox{det}M_0=
1+2is-s^2\beta=q(s),\qquad \beta = 1+2i\xi_0.
$$
We note that $q$ is related to the geometrical spreading
of the beams.
Further manipulations, using the facts that $P^2=P$, $4Q^2=I$ 
and $2(PQ+QP-Q)=\xi_0 I$ reveals that $M(s)$ can be written simply as
$$
  M(s) = \frac{1}{q(s)}\left(M_0- \frac12 s\beta I\right).
$$
We let $m_{ij}$ be the elements of $M$ and deduce that
\begin{align}\label{m11}
m_{11}(s)&=-\frac1{2q(s)}(\xi_0-2i\eta_0^2+s\beta)=-\frac1{2q(s)}(-2i+(\xi_0+s)\beta),\\
\label{m22}
 m_{22}(s)&=
 \frac1{2q(s)}(\xi_0+2i\xi_0^2-s\beta)=
 \frac1{2q(s)}(\xi_0-s)\beta.
\end{align}
Finally, for $a$,
$$
\frac{da}{ds} = -{\rm tr}(M(s))a, \qquad a(0)=a_0.
$$
We note that if $\lambda_0$ and $\lambda_1$ are eigenvalues of $M_0$, then
$q(s) = \det(I+2sM_0)=(1+2s\lambda_0)(1+2s\lambda_1)$ and
$$
{\rm tr} (M(s)) = \frac{\lambda_0}{1+2s\lambda_0}+\frac{\lambda_1}{1+2s\lambda_1}=
\frac12 \frac{d}{ds}\log q(s).
$$
It follows that
\begin{equation}\label{eq:Aexp}
  a(s) = \frac{a_0}{\sqrt{q(s)}}.
\end{equation}
The last thing needed to make the expression
\eqref{eq:gbform} for the Gaussian beam well defined, is 
to decide which $s$-value to use for a given $\x$, i.e.
the function $s^*(\x)$. 
As mentioned above, this is normally taken to
be the $s$-value for
the point on the central ray that is closest to $\x$. Here, however, to
simplify, we just take the $s$-value for the point of the curve that
has the same $y$-coordinate; see Figure~\ref{fig:gbsketch2}.
\begin{figure*}[t]
 \begin{center}
 \includegraphics[width=.3\linewidth]{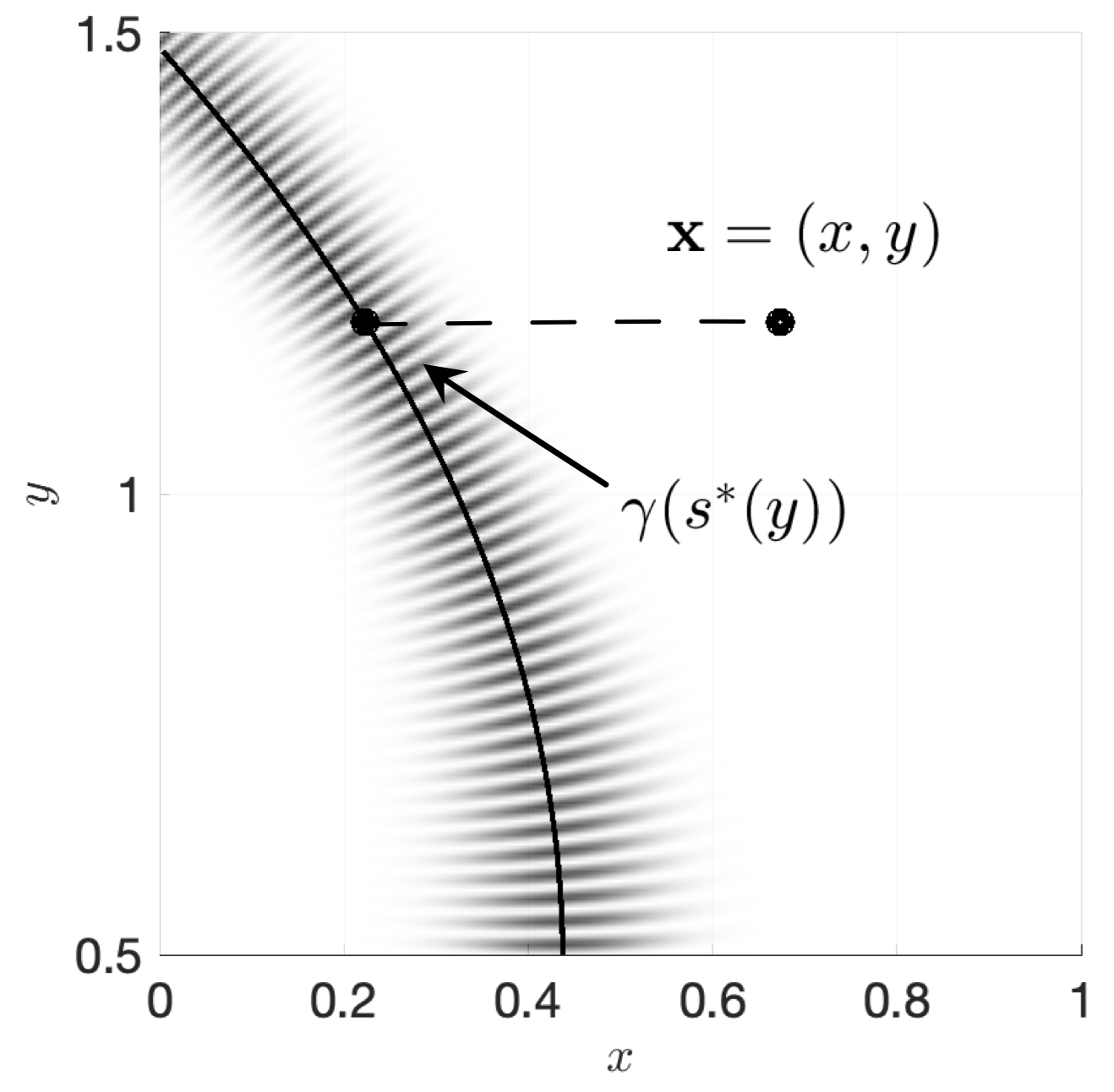}
 \end{center}
 \caption{The simplified map $s^*(y)$.
 }
 \label{fig:gbsketch2}
\end{figure*}

With $\x=(x,y)$ this leads  to
$$
s^*(\x) = s^*(y) =\frac{y - y_0}{2\eta_0}.
$$
Then $\x-\gamma(s^*(\x)) = (x-x(s^*),\, 0)^T$
and \eqref{eq:gbform} becomes
$$
  v_b(x,y) = 
  a\left(s^*\right)\exp\left[
  ik\left(S(s^*) + 
  (x-x(s^*))\xi(s^*) + \frac12m_{11}(s^*)(x-x(s^*))^2
  \right)
  \right],
$$
with $x(s)$, $\xi(s)$, $S(s)$, $m_{11}(s)$, $a(s)$ and $s^*(y)$ given
above.
\begin{rem}
The value of
$q(s)=1+2si-s^2\beta$ crosses the negative real 
axis when $s=1/\xi_0$. 
To find a better branch cut for the square root
in the expression \eqref{eq:Aexp} for $A(s)$
we note first that the equation
$$
\Im \frac{q(s)}{\beta}=0 \quad\Leftrightarrow\quad
   (\Im\beta)(\Re q(s))=(\Re\beta)(\Im q(s))\quad\Leftrightarrow\quad
   2\xi_0(1-s^2) = 2s(1-\xi_0s)
$$
has the unique solution $s=\xi_0$. 
Therefore, the equation $q(s)=t\beta$ with $t\in\Real$ only has 
a solution if $t=q(\xi_0)/\beta=(1-\xi_0^2)>0$. Hence, 
$q(s)$
never crosses the line $\{-t\beta\,:\, t\geq 0\}$, which we
therefore use as branch cut. 
This guarantees a smooth
dependence of the Gaussian beam on $s$ for all $s\geq 0$. 
It can be written as
$\sqrt{z\beta^*}/\sqrt{\beta^*}$ if $\sqrt{\cdot}$ is
the usual square root with branch cut along the negative real axis.
\end{rem}

\subsection{Gaussian beam superposition}\label{GBsuper}

We will now prepare the superposition. 
The initial curve $\Gamma$ is simply the $y$-axis so
that
$\Gamma(z)=\gamma(0;z)=(x_0(z),y_0(z))=(0,z)$.
By assumption
the incoming wave propagates in the same direction 
$(\xi_0,\eta_0)$
for all $z$.
 Moreover, the same initial data for $M$ is used for all $z$.
This means that $x$, $\xi$, $\eta$, $M$ and the 
geometrical spreading parameter $q$ are
all independent of $z$. Only $y$, $S$ and $a$ depend on $z$.
We obtain from \eqref{eq:gbform},
\begin{equation*} 
  v_{\rm beam}(x,y;z) = 
  a(s;z)e^{ik(S(s;z)+(\xsm-\gamma(s;z))\cdot\psm(s)
  +\frac12(\xsm-\gamma(s;z))^TM(s)(\xsm-\gamma(s;z))},
  \qquad s=s^*(\x).
\end{equation*}
To derive the initial data 
$a_0(z)$, $S_0(z)$
for $a(s;z)$ and $S(s;z)$
we consider the trace of $v_{\rm beam}$ 
and $u_{GB}$
on $x=0$.
To get explicit formuale we also 
let $s^*(\x)$ be the $s$-value for
the point on the curve that has the same $x$-value,
i.e. $s^*(0,y)=0$.
That gives
\begin{align*}
  v_{\rm beam}(0,y;z) &= 
  a_0(z)e^{ik[S_0(z)+((0,y)-\gamma(0;z))\cdot\psm_0
  +\frac12((0,y)-\gamma(0;z))^TM(0)((0,y)-\gamma(0;z))]}
  \\
  &=
  a_0(z)e^{ik[S_0(z)+(y-z)\eta_0
  +\frac12(y-z)^2m_{22}(0)]}.
\end{align*}
and for $u_{GB}$, 
\begin{align*}
u_{GB}(0,y) &=  e^{iky\eta_0}
\sqrt{\frac{k}{2\pi}}
\int
  a_0(z)e^{ik[S_0(z)-z\eta_0]+
  \frac12ik(y-z)^2m_{22}(0)}dz\\
  &=
  a_0(y)e^{iky\eta_0}
\sqrt{\frac{k}{2\pi}}\int
  e^{ik[S_0(z+y)-(z+y)\eta_0]+\frac12ikz^2m_{22}(0)}dz
  +O(k^{-1}).
\end{align*}
To match this with the incoming wave
on $x=0$, i.e. $u_{\rm inc}(0,y)=A(y)\exp(ik\eta_0 y)$,
we take
$$  S_0(z) = \eta_0 z,
$$
and
$$
 a_0(y) = A(y) \left(\sqrt{\frac{k}{2\pi}}\int
  e^{\frac12ikz^2m_{22}(0)}dz\right)^{-1}=
A(y)\sqrt{-im_{22}(0)}.
$$
Thus the expressions for the Gaussian beam coefficients are
\begin{subequations}\label{gbcoeffdef}
\begin{align}
 x(s) &= 2s\xi_0-s^2,\\
 y(s;z) &=z+2s\eta_0,\\
 \xi(s) & = \xi_0-s,\\
 \eta(s) &= \eta_0,\\
S(s;z) &= \eta_0 z+2s-2s^2\xi_0+\frac23s^3,\\
    m_{11}(s)&= \frac{2i-(\xi_0+s)\beta}{2q(s)},\\
  a(s;z) &= 
  \frac{A(z)}{\sqrt{q(s)}}\sqrt{-im_{22}(0)}.
\end{align}
\end{subequations}
This gives us the simplified expression for $v_{\rm beam}$,
\be\label{vgbdef}
  v_{\rm beam}(x,y;z) = 
  a\left(s^*;z\right)
  e^{
  ik\left(S(s^*;z) + 
  (x-x(s^*))\xi(s^*) + \frac12m_{11}(s^*)(x-x(s^*))^2
  \right)
  },\quad
  s^*(y;z) =\frac{y - z}{2\eta_0},
\ee
which, together with \eq{gbcoeffdef}, define $u_{GB}$ via
\be\label{ugbsupdef}
u_{GB}(x,y)=  \sqrt{\frac{k}{2\pi}}\int v_{\rm beam}(x,y;z)  dz.
\ee
We will continue now to simplify \eq{vgbdef}
and compute the $k$-scaled Fourier transform of $u_{GB}$ in $y$.
Since $x-x(s^*)=x-2s^*\xi_0-{s^*}^2$
and $\delta = \xi_0^2-x$ by \eqref{deltadef}
we have
$$
  v_{\rm beam}(x,y;z) = 
A(z)f(y-z) e^{ik\eta_0 y},
\qquad
f(2\eta_0s^*) = 
\sqrt{-im_{22}(0)}
\frac{e^{ 
  ik S_g(
  s^*;\delta)}}{\sqrt{q(s^*)}}
$$
where
\begin{align*}
S_g(s^*;\delta)&:=
S(s^*;z) + 
  (x-x(s^*))\xi(s^*) + \frac12m_{11}(s^*)(x-x(s^*))^2
  -\eta_0 y
\\
  &=
\eta_0 (z-y)+2s^*-2{s^*}^2\xi_0+\frac23{s^*}^3 +
  (x-2s^*\xi_0+{s^*}^2)(\xi_0-s^*)+
\frac12 m_{11}(s^*)(x-2s^*\xi_0+{s^*}^2)^2\\
  &=
2\xi_0^2s^*-2{s^*}^2\xi_0+\frac23{s^*}^3 +
  ((s^*-\xi_0)^2-\delta)(\xi_0-s^*)+
\frac12 m_{11}(s^*)((s^*-\xi_0)^2-\delta)^2\\
  &=
\frac23\xi_0^3 +\frac23(s^*-\xi_0)^3 +
  ((s^*-\xi_0)^2-\delta)(\xi_0-s^*)+
\frac12 m_{11}(s^*)((s^*-\xi_0)^2-\delta)^2.
\end{align*}
Then we can write
$$
   u_{GB}(x,y) = 
   \sqrt{\frac{k}{2\pi}}({A}\ast f)(y)e^{ik\eta_0 y}.
$$
Since ${A}\in {\mathcal S}$ the beam $v$ is always integrable
in $z$, so that $u_{GB}$ in \eq{ugbsupdef}
is well-defined. However, 
there is no guarantee that $u_{GB}(x,\cdot)$ is in $L^1$;
in general it is not.
When we compute its Fourier
transform we therefore use Lemma \ref{lem:Fregularize} and \eq{Fdefdist}.
By Lemma \ref{qest} and \ref{m11est} below, 
$q$ is bounded away from zero and $\Im m_{11}$
is strictly positive. Hence, 
 $f\in L^\infty$. Therefore,
\begin{align*}
{ \hat u}_{GB}(x, \eta_0+\cdot)
 &=\lim_{t\to\infty}
 \Fcal_k\left(\psi_t u_{GB}(x,\cdot)
 \exp(-ik\eta_0 \cdot)
 \right)\\
 &=\lim_{t\to\infty}
 \sqrt{\frac{k}{2\pi}} \Fcal_k\left(\psi_t 
({A}\ast f)\right)
 =\lim_{t\to\infty} \Fcal_k(f\psi_t) 
\Fcal_k({A}),
  \end{align*}
  where we can choose $\psi_t(w) :=\psi((w/2\eta_0-\xi_0)/t)$, that is,
  $b_0=1/2\eta_0\neq 0$ and $b_1=-\xi_0$ in Lemma \ref{lem:Fregularize}.
We then compute
\begin{align*}
 \Fcal_k(f\psi_t)(\delta,\eta)
&=\sqrt{\frac{k}{2\pi}}
\int \psi_t(w) f(w) e^{-ik\eta w}dw
= \sqrt{\frac{k}{2\pi}}
\int \psi_t(2\eta_0 s^*) f(2\eta_0s^*) e^{-2ik\eta \eta_0s^*}
2\eta_0 ds^*
\\
&= \sqrt{-im_{22}(0)}\sqrt{\frac{k}{2\pi}}
\int \psi_t(2\eta_0 s^*) 
  \frac{e^{ 
  ik S_g(
  s^*;\delta)}}{\sqrt{q(s^*)}}
e^{-2ik\eta \eta_0s^*}
2\eta_0 ds^*\\
&= \sqrt{-im_{22}(0)}\sqrt{\frac{k}{2\pi}}
\int \psi(\theta/t) 
  \frac{e^{ 
  ik S_g(
  \theta+\xi_0;\delta)}}{\sqrt{q(\theta+\xi_0)}}
e^{-2ik\eta \eta_0(\theta+\xi_0)}
2\eta_0 d\theta
\end{align*}
Here we made the change of variables 
$w=2\eta_0 s^*$ and $\theta = s^*-\xi_0$.
Moreover,
\begin{align*}
S_g\left(\theta+\xi_0;\delta\right)&=
\frac23\xi_0^3 +\frac23\theta^3 
  -(\theta^2-\delta)\theta+
\frac12 m_{11}(\theta+\xi_0)(\theta^2-\delta)^2\\
&=
\frac23\xi_0^3 -\frac13\theta^3 
  +\delta\theta+
\frac12 m_{11}(\theta+\xi_0)(\theta^2-\delta)^2
\\
&=\frac23\xi_0^3+\phi_g(\delta,\theta,\eta) +
2\eta_0\eta\theta,
\end{align*}
where we have introduced the Gaussian beam phase $\phi_g$ as,
\be\label{eq:gbphase}
  \phi_g(\delta,\theta,\eta) = 
  -\frac{1}{3}\theta^3+\theta(\delta-2\eta_0\eta)
  +\frac12m_{11}(\xi_0+\theta)(\theta^2-\delta)^2.
\ee
Since $m_{22}(0)= \xi_0\beta/2=$ by
\eqref{m22} and $q(\xi_0)=\beta\eta_0^2$,
this finally gives
\begin{align*}
 \Fcal_k(f\psi_t)(\delta,\eta)
&= \sqrt{-im_{22}(0)}\sqrt{\frac{k}{2\pi}}
\int \psi(\theta/t) 
  \frac{e^{ 
  ik 
  (\frac23\xi_0^3+\phi_g(\delta,\theta,\eta) 
-2\eta \eta_0\xi_0
)}}{\sqrt{q(\theta+\xi_0)}}
2\eta_0 d\theta\\
&= 2\eta_0\sqrt{-im_{22}(0)}
\sqrt{\frac{k}{2\pi}}
e^{ 
  ik 
  (\frac23\xi_0^3
-2\eta \eta_0\xi_0
)}
\int \psi(\theta/t) 
  \frac{e^{ 
  ik 
 \phi_g(\delta,\theta,\eta) }}{\sqrt{q(\theta+\xi_0)}}
d\theta\\
&= 
\sqrt{-2i\xi_0q(\xi_0)}
\sqrt{\frac{k}{2\pi}}
e^{ 
  ik 
  (\frac23\xi_0^3
-2\eta \eta_0\xi_0
)}
\int \psi(\theta/t) 
  \frac{e^{ 
  ik 
 \phi_g(\delta,\theta,\eta) }}{\sqrt{q(\theta+\xi_0)}}
d\theta\\
&=: 
k^{1/6} P_{GB}(k,\eta) I_t(\eta,\delta,k),
\end{align*}
where
\begin{equation}\label{pgbdef}
P_{GB}(k,\eta)=
2
(\pi \xi_0)^{1/2}
e^{-i\pi/4}
e^{ 
  ik (\frac23\xi_0^3
-2\eta \eta_0\xi_0
)},
\end{equation}
and
\begin{equation}\label{Itdef}
I_t(\eta,\delta,k) =  \frac{k^{1/3}\sqrt{q(\xi_0)}}{2\pi}
\int \psi(\theta/t) 
  \frac{e^{ 
  ik 
 \phi_g(\delta,\theta,\eta) }}{\sqrt{q(\theta+\xi_0)}}
d\theta.
\end{equation}
Then
\begin{equation}\label{vgbhatdef}
{ \hat u}_{GB}(x, \eta_0+\eta)=
{\hat v}_{GB}(\eta,x,k)
 {\hat{A}}(\eta),
\end{equation}
with
$$
{\hat v}_{GB}(\eta,x,k)=\lim_{t\to \infty}\ k^{1/6} P_{GB}(k,\eta)
I_t(\eta,\delta,k).
$$


\section{Properties of the amplitude and phases}
\label{sec:ampphase}


In this section we collect a series of estimates
that we will need for the final proof of the magnitude
of the Gaussian beam error.

\subsection{Geometrical spreading}

Here we show some properties of the
$q$-polynomial in \eqref{qdef} that relates to the
geometrical spreading, repeated here for convenience,
$$
  q(s)= 1 +2is-\beta s^2, \qquad \beta = 1+2i\xi_0.
$$
We have
\begin{lem}\label{qest}
There are positive constants $q_0$ and $q_1$,
independent of $\xi_0$
and $\theta\in\Real$,
such that
$$
   0<q_0(1+\theta^2) \leq |q(\xi_0+\theta)|\leq q_1(1+\theta^2).
$$
Furthermore, there are constants $b_n$ independent of $\theta$ and $\xi_0$
such that
$$
  \left|\frac{d^n}{d\theta^n}\frac{1}{\sqrt{q(\xi_0+\theta)}}\right|\leq
  \frac{b_n}{1+|\theta|^{n+1}}.
$$

\end{lem}
\begin{proof}
We first show that $q$ has no real root for the
considered values of $\xi_0$.
Suppose therefore that $q$ has a real root $s=r$. Then
the real and imaginary parts of $q(r)=0$ reads
$$
  1-r^2=0,\qquad 2r-2\xi_0r^2=0.
$$
The only solutions to this system are $r=\xi_0=\pm 1$,
which are both ruled out by \eqref{xi0deltabound}.
Let $\tilde{q}(\theta):=|q(\xi_0+\theta)|/(1+\theta^2)$,
which is then continuous and non-zero for all $\theta$.
For large $\theta$ it is bounded from below and above since
$\lim_{\theta\to\pm \infty}\tilde{q}(\theta)=|\beta|$. In fact, there is a
a constant
$\theta_0$ such that
$$
    \left|\tilde{q}(\theta)-|\beta|\right| \leq \frac{|\beta|}{2},\qquad |\theta|>\theta_0,
$$
uniformly in $\xi_0$, because of the bound \eqref{xi0deltabound}.
We can then take $q_0$ and $q_1$ as 
$$
q_0 = \min\left(\inf_{|\theta|\leq \theta_0} \tilde{q}(\theta),\frac{|\beta|}{2}\right),\qquad
q_1 = \max\left(\sup_{|\theta|\leq \theta_0} \tilde{q}(\theta),\frac{3|\beta|}{2}\right).
$$
The stated bound then follows.

For the second statement, we observe that
there exists a sequence of polynomials $p_n$ of
degree $n$ such that
$$
  \frac{d^n}{d\theta^n}\frac{1}{\sqrt{q(\theta)}}=
  \frac{p_n(\theta)}{q(\theta)^{n+1/2}},
$$
given by the recursion
$$p_{n+1}(\theta)=
p_n'(\theta)q(\theta)-(n+1/2)p_n(\theta)q'(\theta),
$$
thanks to
$$
  \frac{d}{d\theta}\frac{p_n(\theta)}{q(\theta)^{n+1/2}}=
  \frac{p_n'(\theta)q(\theta)-(n+1/2)p_n(\theta)q'(\theta)}{q(\theta)^{n+3/2}}.
$$
Then, by the lower bound on $|q|$ and \eqref{xi0deltabound},
there are constants $C_n$ and $b_n$ depending on $n$ but independent of $\theta$
and $\xi_0$
such that
$$
  \left|\frac{d^n}{d\theta^n}\frac{1}{\sqrt{q(\xi_0+\theta)}}\right|\leq
  \frac{|p_n(\xi_0+\theta)|}{q_0^{n+1/2}(1+\theta^2)^{n+1/2}}
  \leq C_n
  \frac{1+|\xi_0+\theta|^n}{1+|\theta|^{2n+1}}
  \leq C_n
  \frac{1+2^{n-1}(|\xi_0|^n+|\theta|^n)}{1+|\theta|^{2n+1}}
  \leq b_n
  \frac{1}{1+|\theta|^{n+1}}.
$$
This shows the lemma.
\end{proof}

\subsection{Phase}

In this section we define the two
phase functions $\phi_a$ and $\phi_g$ that turn up in our analysis
and show a few properties of them.
The first one, $\phi_a$, is defined as
\begin{equation}\label{phiadef}
  \phi_a(\delta,\theta,\eta) = 
  -\frac{1}{3}\theta^3+\theta(\delta-2\eta_0\eta).
\end{equation}
This has a close connection to the Airy function
and we call it the Airy phase.
Indeed, $\exp(i\theta^3/3)$ is the
Fourier transform of $\Ai(z)$ in ${\mathcal S}'(\Real)$
and therefore, using the regularization of Lemma~\ref{lem:Fregularize},
$$
\lim_{s\to\infty}\frac{k^{1/3}}{2\pi}
 \int \psi(\theta/s)
e^{-ik\phi_a(\delta,\theta,\eta)}
d\theta =
\Ai(k^{2/3}(2\eta_0\eta-\delta)).
$$
The second phase is the
Gaussian beam phase \eqref{eq:gbphase} derived in the previous
section. It can be written as
a sum of $\phi_a$ and a complex correction, by \eqref{m11},
\be\label{pgpadef}
  \phi_g(\delta,\theta,\eta) = 
  \phi_a(\delta,\theta,\eta)
  +\frac12m_{11}(\xi_0+\theta)(\theta^2-\delta)^2, \qquad
  m_{11}(s) = \frac{2i-(\xi_0+s)\beta}{2q(s)}.
\ee
We start by looking at the $m_{11}$ part of the 
Gaussian beam phase.
\begin{lem}\label{m11est} 
For $m_{11}$ it holds that
\begin{align}
  \lim_{s\to \pm\infty} m_{11}(s)s &= 
  \lim_{s\to \pm\infty} -m'_{11}(s)s^2 = \frac12,\label{m11b}\\
  \Im m_{11}(s) &= \frac{\eta_0^2}{|q(s)|^2},\label{m11a}\\
  \left|\frac{d^nm_{11}(\xi_0+\theta)}{d\theta^n}\right|
  &\leq\frac{d_n}{1+|\theta|^{n+1}},\label{m11c} \\
  \left|m_{11}(\xi_0+\theta)\right|
  &\geq\frac{D_0}{1+|\theta|},\label{m11d}
\end{align}
for some constants $d_n$, $D_0$ independent of $\theta$ and $\xi_0$.
\end{lem} 
\begin{proof}
From \eqref{pgpadef} we have
$$
 sm_{11}(s) = s\frac{2i-(\xi_0+s)\beta}{2q(s)}
 =\frac{(2i-\xi_0\beta)s-\beta s^2}{2(1+2i-\beta s^2)}
 =\frac{(-2i+\xi_0\beta)s^{-1}+\beta }{-2(1+2i)s^{-1}+2\beta }\to \frac12,
$$
showing the first limit in \eqref{m11b}.
Moreover,
$$
  m_{11}'(s) = 
  -\frac12\frac{\beta(1+2is-s^2\beta) + (2i-2s\beta)(2i-(\xi_0+s)\beta)
  }
  {q(s)^2}
  =
  -\frac12\frac{
  \beta-4  -2i\beta\xi_0
  +2s\beta\Bigl(\xi_0\beta
  -2i\Bigr)
  +s^2\beta^2  
  }
  {q(s)^2}
$$
which similarly implies the second limit in \eqref{m11b}. 
For \eqref{m11a} we have by
(\ref{m11}),
 \begin{align*}
  \Im m_{11}(s) &= \frac{\Im\left[(2i-(\xi_0+s)\beta)
  (1-2is-s^2\beta^*)
  \right]}{2|q(s)|^2}
 =
  \frac{(\xi_0+s)(2s-2s^2\xi_0) +2(\eta_0^2-s\xi_0)(1-s^2)}{2|q(s)|^2}\\
  &=\frac{\eta_0^2+s\left[(\xi_0+s)(1-s\xi_0) -\eta_0^2s-\xi_0(1-s^2)\right]}{|q(s)|^2}
  = \frac{\eta_0^2}{|q(s)|^2}>0.
\end{align*}
To show \eqref{m11d} with $D_0 = \min(1/2,1-\bar\xi_1)/\sqrt{2}q_1$
we observe that, 
\begin{align*}
|m_{11}(\xi_0+\theta)|&=\frac{|2i-(\theta+2\xi_0)(1+2i\xi_0)|}{|q(\theta+\xi_0)|}
\geq 
\frac{|\theta+2\xi_0|+2|1-(\theta+2\xi_0)\xi_0|}{\sqrt{2}q_1(1+\theta^2)}
\\
&\geq\frac{
|\theta+2\xi_0|+
1-|(\theta+2\xi_0)\xi_0|
}{\sqrt{2}q_1(1+\theta^2)}
=\frac{
1+
(1-\xi_0)|\theta+2\xi_0|
}{\sqrt{2}q_1(1+\theta^2)}
\geq\frac{
1-2(1-\xi_0)\xi_0
+(1-\bar\xi_1)|\theta|
}{\sqrt{2}q_1(1+\theta^2)}
\geq\frac{
D_0
}{1+|\theta|},
\end{align*}
where we used Lemma~\ref{qest} as well as the facts that $\sqrt{2}|z|\geq |\Re z|+|\Im z|$ and 
$1-2x(1-x)\geq 1/2$ for $0\leq x\leq 1$.

For \eqref{m11c} we note that
if $r_k$ and $p$ are any polynomials of degrees $\ell_k$ and $\ell$, then
$$
  \frac{d}{d\theta}\left(\frac{r_k}{p^k}\right)
  = 
  \frac{r_k'p  -r_kp'}{p^{k+1}}=:
  \frac{r_{k+1}}{p^{k+1}},
$$
where the degree of $r_{k+1}$ is
 $\ell_{k+1}=\ell_k+\ell-1$. By induction,
$$
  \frac{d^n}{d\theta^n}\left(\frac{r_1}{p}\right) = \frac{r_{n+1}}
  {p^{n+1}},
$$
and $r_{n+1}$ has degree 
 $\ell_{n+1} = \ell_1+n(\ell-1)$.
Since $m_{11}$ is the quotient of the 
first order polynomial $r_1(s)=(2i-\xi_0\beta - \beta s)/2$,
and the second order polynomial $q(s)$,
its $n$-th derivative, is 
$$
m^{(n)}_{11}(\xi_0+\theta)=\frac{r_{n+1}(\xi_0+\theta)
  }{q(\xi_0+\theta)^{n+1}}
$$
and $r_{n+1}$ is of degree $n+1$.
Using Lemma~\ref{qest} 
and \eqref{xi0deltabound}
we then obtain the required estimate,
$$
|m^{(n)}_{11}(\xi_0+\theta)|
  \leq \frac{C\left(1+|\xi_0+\theta|^{n+1}\right)
  }{(q_0(1+\theta^2))^{n+1}}
  \leq \frac{C\left(1+2^n(|\xi_0|^{n+1}+|\theta|^{n+1})\right)
  }{(q_0(1+\theta^2))^{n+1}}
  \leq \frac{d_n}{1+|\theta|^{n+1}},
$$
where $d_n$ is independent of $\xi_0$.
\end{proof}

We are now ready to estimate the full phases $\phi_a$ and $\phi_g$.
\begin{lem}\label{phiest} 
Let $\phi$ be either $\phi_a$ or $\phi_g$. 
Suppose $|\delta|\leq 1$.
Then
there are constants $c_0$ and $C_n$,
independent
of $\eta$,
$\theta$, $\xi_0$ and $\delta$ such that
\begin{align}
|\phi_\theta(\delta,\theta,\eta)|&\geq \frac{\theta^2}{16},\quad 
\text{when $|\theta|\geq c_0\left(1+|\eta|^{1/2}\right)$},\label{phiestc}\\
|\partial_\theta^n\phi(\delta,\theta,\eta)| &\leq 
C_n \begin{cases}
|\theta|^{2}+\delta+|\eta|, & n=1, \\
|\theta|+\delta, & n=2, \\
\frac{1}{1+|\theta|^{n-3}}, & n\geq 3.
\end{cases}\label{phiestd}
\end{align}
Additionally,
\begin{equation}\label{phiesta}
\Im \phi(\delta,\theta,\eta) \geq 0. 
\end{equation}
For $\phi_a$ we have 
$C_n=0$ when $n\geq 4$.
\end{lem} 
\begin{proof}

We first prove the statements for 
$\phi=\phi_a$. 
Suppose $|\theta|\geq c_0(1+|\eta|)^{1/2}$.
Then $|\eta| \leq\theta^2/c_0-1$ and by \eqref{xi0deltabound}
$$
|\partial_\theta\phi_a(\delta,\theta,\eta)| = 
|\theta^2+2\eta\eta_0-\delta|\geq \theta^2 - 2|\eta|-1\geq
\theta^2 - 2(\theta^2/c_0^2-1)-1\geq 
(1-2/c_0^2)\theta^2,
$$
which gives \eqref{phiestc} when $c_0\geq \sqrt{32/15}$.
Similarly,
$$
|\partial_\theta\phi_a(\delta,\theta,\eta)| \leq
\theta^2 + 2|\eta|+\delta\leq
2(\theta^2 + |\eta|+\delta),
$$
showing \eqref{phiestd} for $n=1$. The bounds for larger $n$ follow easily from
an explicit calculation, yielding $C_n= 2$ for $1\leq n\leq 3$
and $C_n=0$ for $n\geq 4$.

To prove the statements for $\phi_g$ we
denote the correction term by $w(\delta,\theta)=\phi_g-\phi_a$.
Lemma~\ref{m11est} gives
\begin{align*}
\lim_{\theta\to\pm\infty} \frac{ w_\theta(\delta,\theta)}{\theta^2}
&=
\lim_{\theta\to\pm\infty}\frac{1}{\theta^2}\left(
\frac12m'_{11}(\xi_0+\theta)(\theta^2-\delta)^2
  +2m_{11}(\xi_0+\theta)(\theta^2-\delta)\theta\right)
  = -\frac14+1=\frac34.
\end{align*}
Consequently, there is a $K$ 
such that 
$|w_\theta(\delta,\theta)|\leq 7\theta^2/8$ for all $|\theta|\geq K$,
uniformly in $\xi_0$ and $\delta$ thanks to
\eqref{xi0deltabound}. 
We now take $c_0=\max(K,\sqrt{32})$. Then for 
$|\theta|\geq c_0(1+|\eta|^{1/2})\geq K$, we have
$$
|\partial_\theta\phi_g(\delta,\theta,\eta)|
\geq |\partial_\theta\phi_a(\delta,\theta,\eta)|-|\partial_\theta w(\delta,\theta)|
\geq 
(1-2/c_0^2)\theta^2-\frac78\theta^2
=(1/8-2/c_0^2)\theta^2\geq \frac1{16}\theta^2,
$$
and \eqref{phiestc} is proved.

For \eqref{phiestd} we use \eqref{m11c} in Lemma~\ref{m11est}.
When $n=1$ we have, as above,
\begin{align*}
|\partial_\theta\phi_g(\delta,\theta,\eta)|
&\leq
  \theta^2 +|\delta-2\eta_0\eta|
  +\frac12|m'_{11}(\xi_0+\theta)|(\theta^2-\delta)^2
  +2|m_{11}(\xi_0+\theta)|\theta(\theta^2-\delta)
\\
&\leq
  \theta^2 +|\delta-2\eta_0\eta|
  +\frac{d_1}{2}\frac{|\theta|^4+\delta^2}{1+\theta^2}
  +2d_0\frac{|\theta|^3+\delta|\theta|}{1+|\theta|}
\leq \left(1+\frac12d_1+2d_0\right)(\theta^2 + \delta) +2\eta_0|\eta|,
\end{align*}
which shows the result for $n=1$ with $C_1=\max(1+d_1/2+2d_0,2)$.
For $n=2$ we get
\begin{align*}
  |\partial_{\theta\theta}\phi_{g}(\delta,\theta,\eta)| &= 
  \left|-2\theta
  +\frac12m_{11}^{(2)}(\xi_0+\theta)(\theta^2-\delta)^2
  + 4m_{11}^{(1)}(\xi_0+\theta)(\theta^2-\delta)\theta
  +2m_{11}(\xi_0+\theta)(3\theta^2-\delta)
  \right|\\
&\leq
  2|\theta|
  +\frac{d_2}{2}\frac{|\theta|^4+\delta^2}{1+|\theta|^3}
  + 4d_1\frac{|\theta|^3+\delta|\theta|}{1+\theta^2}
  +2d_0\frac{3\theta^2+\delta}{1+|\theta|} \\
&\leq \left(2+\frac12d_2+4d_1+6d_0\right)(|\theta|+\delta)
=: C_2(|\theta|+\delta).
\end{align*}
For $n\geq 3$,
\begin{align*}
  |\partial^n_\theta w(\delta,\theta)| &= 
  \frac12\left|\sum_{\ell=0}^n\Vector{n\\ \ell}
  m_{11}^{(n-\ell)}(\xi_0+\theta)
  \frac{d^\ell}{d\theta^\ell}(\theta^2-\delta)^2
  \right|=
  \frac12\left|\sum_{\ell=0}^{\max(4,n)}\Vector{n\\ \ell}
  m_{11}^{(n-\ell)}(\xi_0+\theta)
  \frac{d^\ell}{d\theta^\ell}(\theta^2-\delta)^2
  \right|\\
&\leq
  C\sum_{\ell=0}^{\max(4,n)}\Vector{n\\ \ell}
  \frac{1}{1+|\theta|^{n-\ell+1}}(1+|\theta|^{4-\ell})
\leq \frac{C_n}{1+|\theta|^{n-3}},
\end{align*}
which 
shows the result for $n\geq 3$
 as $\partial^3_\theta\phi_g=-1+\partial^3_\theta w$
and $\partial^n_\theta\phi_g=\partial^n_\theta w$ for $n\geq 4$.

Finally, statement \eqref{phiesta} for $\phi_a$ is trivial, as $\Im\phi_a=0$,
and for $\phi_g$
it follows 
from \eqref{m11a} in Lemma~\ref{m11est} and \eqref{xi0deltabound},
since
\begin{align*}
  \Im \phi_g(\delta,\theta,\eta)&=
\frac{1}{2}\Im m_{11}(\xi_0+\theta)(\theta^2-\delta)^2=
\frac{(\theta^2-\delta)^2\eta_0^2}
{2|q(\xi_0+\theta)|^2}\geq 0.
\end{align*}
This concludes the proof.
\end{proof}

In the final part of this section
we consider a space of function that
is used in Lemma~\ref{nonstatphase}.
For a fixed phase function $\phi$ and order $p$ we
first introduce the basis functions
\begin{equation}\label{Wdef}
  {\mathcal W}_p(\phi)=\left\{
  \prod_{k}\frac{\phi^{(\alpha_k+1)}}{\phi'}\ : \ 
  \sum_{k}\alpha_k=p,\ \alpha_k\geq 1,\right\},
\end{equation}
when $p\geq 1$ and let ${\mathcal W}_0(\phi)$ be the constant function
equal to one.
Second, we denote by  ${\mathcal U}_p(\phi)$ the linear span of these functions
over the complex numbers,
\begin{equation}\label{Udef}
  {\mathcal U}_p(\phi) = \operatorname{span}_{\Cnumbers}  {\mathcal W}_p(\phi).
\end{equation}
Functions in ${\mathcal U}_p(\phi)$ appear in Lemma~\ref{nonstatphase}.
Here we show that
when $\phi$ is 
either $\phi_a$ or $\phi_g$,
these functions are bounded on subsets where the phase gradient
grows at least quadratically.
\begin{lem}\label{Uest} 
Let $K=\{ \theta\in\Real\ |\ r_0\leq |\theta| \leq r_1 \}$,
with $0<R_0\leq r_0<r_1$ and,
for some $c>0$, 
$$
  |\phi_\theta(\delta,\theta,\eta)|\geq c \theta^2,
  \qquad \text{for all $\theta\in K$ and $|\delta|\leq 1$},
$$
where $c$ and $R_0$ are independent of $\delta$ and $\eta$.
Then 
for each $u\in {\mathcal U}_p(\phi(\delta,\cdot,\eta))$,
where $\phi$ is either $\phi_a$ or $\phi_g$,
there is a uniform bound 
$$
  |u(\theta)|\leq C,
  \qquad \forall \theta\in K.
$$
The constant $C$ 
depends on 
$c$ and $R_0$ 
but
is independent of $r_0$, $r_1$, $\delta$ and $\eta$.

\end{lem} 
\begin{proof}
We get from 
Lemma~\ref{phiest},
$$
|\phi_{\theta\theta}(\delta,\theta,\eta)|
\leq C_2 (|\theta|+\delta)
\leq C_2(|\theta|+1)
\leq C_2(1+1/R_0)|\theta|,
\qquad \forall \theta\in K,
$$
while for $n\geq 3$ we have
$$
|\partial_\theta^n\phi(\delta,\theta,\eta)|
\leq 
\frac{C_n}{1+|\theta|^{n-3}}
\leq C_n,
\qquad \forall \theta\in \Real.
$$
The constants $C_2$ and $C_n$ are independent of
$\delta$ and $\eta$.
Consider next $w\in {\mathcal W}_p(\phi(\delta,\cdot,\eta))$ and assume it has $M\geq 1$ factors
and, without loss of generality, that the first $M'\leq M$ factors
have $\alpha_k=1$.
Then, for all
$\theta\in K$,
\begin{align*}
|w(\theta)| &=
\prod_{k=1}^{M'}
\frac{|\phi_{\theta\theta}(\delta,\theta,\eta)|}
           {|\phi_\theta(\delta,\theta,\eta)|}
\prod_{k=M'+1}^M
\frac{|\partial_\theta^{\alpha_k+1}\phi(\delta,\theta,\eta)|}
           {|\phi_\theta(\delta,\theta,\eta)|}
\leq 
\prod_{k=1}^{M'}
\frac{
C_2(1+1/R_0)|\theta|}
           {c\theta^2}
\prod_{k=M'+1}^M
\frac{C_{\alpha_k+1}}
           {|c\theta^2|}
           \\
&\leq
\left(\max_{2\leq j \leq p+1} C_j\right)^M
 \frac{(1+1/R_0)^{M'}}{c^M|\theta|^{2M-M'}} 
\leq C \frac{(1+1/R_0)^{M'}}{R_0^{2M-M'}} 
=: \tilde{C},
\end{align*}
where $\tilde{C}$ is independent of $\delta$, $\eta$, $r_0$ and $r_1$,
but depends on $c$ and $R_0$.
Since $u\in {\mathcal U}_p(\phi(\delta,\cdot,\eta))$
is a linear combination of functions in ${\mathcal W}_{p}
(\phi(\delta,\cdot,\eta))$
the same bound holds true for $u$ on $K$.
\end{proof}

\section{Estimates of oscillatory integrals}
 
We consider integrals of the type
$$
k^{1/3}
 \int\theta^p
 r(\theta)
e^{ik\phi(\delta,\theta,\eta)}
d\theta,
$$
where $\phi$ is either $\phi_a$ or $\phi_g$
and $r\in W^{n,\infty}(\Real)$, whose norm is defined by
$$
  ||r||_{W^{n,\infty}(\Real)} = \sum_{j=0}^n \left\|\frac{d^jr}{d\theta^j}\right\|_{L^\infty(\Real)}
$$
In general the integrand is then not in $L^1(\R)$ and the 
integral must be defined in a generalized sense as an
oscillatory integral. In this section, however, we
only estimate the integral over bounded intervals that
are defined using a smooth cutoff function
$\psi\in C_c^\infty(\Real)$, which
takes values in $[0,1]$, is equal to one on $[-1,1]$ and has supp$(\psi)\subset[-2,2]$. 
This leaves us with integrals over compact domains with
smooth integrands. Our main tool
for estimating them are the identities stated in Lemma~\ref{nonstatphase}.
They are used to rewrite the integral on the domain where
$\phi_\theta\neq 0$. 
Lemma~\ref{phiest} in the previous
section tells us when this is true.
Lemma~\ref{nonstatphase} 
uses the space of functions ${\mathcal U}_{p}$ defined 
in (\ref{Udef}). Functions in ${\mathcal U}_{p}$
are bounded on the domains we consider here, which is proved in
Lemma~\ref{Uest}. 
Together, Lemma~\ref{nonstatphase} and Lemma~\ref{Uest}
constitute a precise version of the non-stationary
phase lemma.

We will also make use of the simple inequalities
\begin{equation}
  ||uv||_{W^{n,\infty}(\Real)}\leq C_n
  \sum_{j=0}^n 
    \sum_{k=0}^j \left\|\frac{d^ku}{d\theta^k}\right\|_{L^\infty(\Real)}
    \left\|
    \frac{d^{j-k}v}{d\theta^{j-k}}\right\|_{L^\infty(\Real)}
    \leq C_n
  ||u||_{W^{n,\infty}(\Real)}
  ||v||_{W^{n,\infty}(\Real)},
  \label{Wnineq1}
\end{equation}
for all $u,v\in W^{n,\infty}(\Real)$,
and
\begin{equation}
  ||u(\cdot/\sigma)||_{W^{n,\infty}(\Real)}=
  \sum_{j=0}^n \sigma^{-j}\left\|\frac{d^ju}{d\theta^j}\right\|_{L^\infty(\Real)}
  \leq   \max(1,\sigma^{-n})||u||_{W^{n,\infty}(\Real)}, \qquad
  \forall u\in W^{n,\infty}(\Real),\ \sigma\neq 0.
  \label{Wnineq2}
\end{equation}

We start with an estimate of the integral
between $R$ and $2t$ where $t$ is 
arbitrarily large. For this
we consider a smooth cutoff around
$\theta=R$ and $\theta=t>R$ and 
obtain bounds that are independent of $t$.

\begin{lem}\label{Itail} 
Let $\phi$ be either $\phi_a$ or $\phi_g$ and set
$$
I_t = k^{1/3}
 \int (1-\psi(\theta/R))\psi(\theta/t)\theta^p
 r(\theta)
e^{ik\phi(\delta,\theta,\eta)}
d\theta,
$$
where $c_0\leq R < t$ with $c_0$ as in Lemma~\ref{phiest},
$|\delta|\leq 1$
and $r\in W^{n,\infty}(\Real)$.
If $c_0(1+|\eta|^{1/2})\leq R$
and $n\geq 1+p/2$,
there is a constant $C_n$ independent of $k$, $R$, $\delta$, $\eta$ and $t$
such that
$$
 |I_t|\leq  
C_nk^{1/3-n}
  ||r||_{W^{n,\infty}(\Real)}.
$$
\end{lem} 
\begin{proof}
On this domain the results
in Section~\ref{sec:ampphase}
 show that the phase gradient
does not vanish, and $|\phi_\theta|\geq c\theta^2$. Since the
integrand is smooth and compactly supported we can therefore
use the non-stationary phase lemma to estimate the integral.
For sufficiently regular $r$
the repeated partial integrations in this lemma
enables us to offset the growing $\theta^p$ factor and obtain a
bound that is independent of $t$.

To be precise, let 
$$
  b(\theta) = 
(1-\psi(\theta/R))
\psi(\theta/t)
 r(\theta),
 $$
 which is supported in the compact set
 $K=\{ \theta\in\Real\ |\ R\leq |\theta| \leq 2t \}$.
 Then 
 by Lemma~\ref{phiest} we have 
$|\phi_\theta(\delta,\theta,\eta)|\geq \theta^2/16$
on $K$, independent of $\delta$, $\eta$ and $t$,
 since $|\theta|
\geq R>c_0(1+|\eta|^{1/2})$.
We apply Lemma~\ref{nonstatphase} with 
 $a(\theta)=b(\theta)\theta^p$  
 and let $D$ be any bounded open set containing $K$.
This gives
 \begin{align}
 I_t&=k^{1/3} \int_{D} b(\theta)\theta^p 
 e^{ik\phi(\delta,\theta,\eta)}
d\theta  \nonumber\\
&= 
  k^{1/3}(ik)^{-n}
\sum_{\ell=0}^n 
\int_{K} \left(\frac{d^\ell}{d\theta^\ell}b(\theta)\theta^p \right)
  \frac{u_{\ell,n}(\theta)}{{\phi_\theta(\delta,\theta,\eta)}^{n}}
   e^{ik\phi(\delta,\theta,\eta)}
d\theta,\qquad
u_{\ell,n} \in {\mathcal U}_{n-\ell}(\phi(\delta, \cdot, \eta)).\label{Inonstat1}
\end{align}
where the space ${\mathcal U}_{p}$ is defined in \eqref{Udef}.
Since $K$ satisifies all conditions in
Lemma~\ref{Uest} and $R\geq c_0>0$
we obtain a uniform bound, 
$$
  |u_{\ell,n}(\theta)|\leq C_{\ell,n},\qquad \forall \theta\in K,
$$
where $C_{\ell,n}$ depends on $c_0$, but
is
independent of
$\delta$, $\eta$, $R$ and $t$.
This allows us to estimate $I_t$ as
\begin{align*}
 |I_t| &\leq  
  Ck^{1/3-n}
\sum_{\ell=0}^n 
\sum_{j=0}^{\min(\ell,p)}
\int_{K}  |b^{(\ell-j)}(\theta)|
  \frac{|\theta|^{p-j}}{|\theta|^{2n}}d\theta
\leq
  Ck^{1/3-n}
\sum_{\ell=0}^n 
\sum_{j=0}^{\min(\ell,p)}
  ||b^{(\ell-j)}||_{L^\infty(\Real)}
\int_{R}^{\infty} \frac{d\theta}{
|\theta|^{2n-p+j}} 
\\
&\leq
Ck^{1/3-n}
  \sum_{\ell=0}^n 
||b^{(\ell)}||_{L^\infty(\Real)}
  \sum_{j=0}^p
  \frac{1}{R^{2n-p+j+1}}
\leq  C\frac{k^{1/3-n}}{c_0^{2n-p+1}}
  ||b||_{W^{n,\infty}(\Real)},
\end{align*}
where we also used the fact that $2(n-1)\geq p$ and $c_0\leq R\leq t$.
Moreover, by \eq{Wnineq1} and \eq{Wnineq2},
$$
||b||_{W^{n,\infty}(\Real)}
\leq C_n^2\max(1,c_0^{-n})^2
||1-\psi||_{W^{n,\infty}(\Real)}
||\psi||_{W^{n,\infty}(\Real)}
||r||_{W^{n,\infty}(\Real)}
\leq C
||r||_{W^{n,\infty}(\Real)}
$$
The result in the lemma follows.
\end{proof}

Next we consider the main part of the integral
for small $\eta$ and $\delta$ with the Airy phase.
The estimate involves the norm of $r$ with an argument scaled
by $k^{1/3}$. 

\begin{lem}\label{Imainpartairy} 
Let
$$
I=k^{1/3}
 \int\psi(\theta/R)\theta^p
 r(\theta)
e^{ik\phi_a(\delta,\theta,\eta)}
d\theta,
$$
where $0<R_0\leq R\leq R_1$
and $r\in W^{n,\infty}(\Real)$
with $n=\lceil (p+1)/5\rceil$.
Moreover, suppose 
\begin{equation}\label{etadeltasmall}
2|\eta|+|\delta|\leq \frac{\xi_0^2}2k^{-2/3}, \qquad k\geq 1.
\end{equation}
Then there is a constant $C$, depending on $R_0$ and $R_1$, but independent of $k$, $R$, $\eta$ and $\delta$
such that
$$
 |I|\leq  Ck^{-p/3}||\tilde{r}_k||_{W^{n,\infty}(\Real)}, \qquad
 \tilde{r}_k(\zeta) := r\left(k^{-1/3}\zeta\right),
$$
\end{lem} 
\begin{proof}
For $\delta-2\eta\eta_0=0$
the phase $\phi_a$ has a degenerate stationary point at the origin.
We will therefore treat 
the integral in the vicinity of the origin
 separately. Away from the origin
we have the same type of lower bound $\partial_\theta\phi_a\geq c\theta^2$
as in Lemma~\ref{Itail} and we can therefore once again use 
the non-stationary phase lemma to estimate
the integral.

For the proof we use the rescaled variables
$\zeta=k^{1/3}\theta$,
$\tilde{\delta}=k^{2/3}\delta$,
$\tilde{\eta}=k^{2/3}\eta$.
Since
$$
  k\phi_a(\tilde\delta/k^{2/3},\zeta/k^{1/3},\tilde\eta/k^{2/3}) = 
  -\frac{1}{3}\zeta^3+\zeta(\tilde\delta-2\eta_0\tilde\eta)
  =\phi_a(\tilde\delta,\zeta,\tilde\eta),
$$
we can rewrite the integral as
$$
I=
    k^{-p/3}
 \int \psi\left(\frac{\zeta}{k^{1/3}R}\right)\zeta^p
  \tilde{r}_k(\zeta)
e^{i\phi_a\left(\tilde\delta,\zeta,\tilde\eta\right)}
d\zeta.
$$
We then divide the integral into two pieces,
\begin{align*}
I&=
    k^{-p/3}
 \int \psi\left(\zeta\right)
 \psi\left(\frac{\zeta}{k^{1/3}R}\right)\zeta^p
   \tilde{r}_k(\zeta)
e^{i\phi_a\left(\tilde\delta,\zeta,\tilde\eta\right)}
d\zeta \\
&\ \ +
    k^{-p/3}
 \int 
 (1-\psi\left(\zeta\right))
 \psi\left(\frac{\zeta}{k^{1/3}R}\right)\zeta^p
  \tilde{r}_k(\zeta)
e^{i\phi_a\left(\tilde\delta,\zeta,\tilde\eta\right)}
d\zeta = k^{-p/3}(I_1+I_2),
\end{align*}
where $I_1$ is the part close to the origin containing
the stationary point, and $I_2$ is the remaining part.
Note that $I_2$ matches the general
form of the integral 
in Lemma~\ref{Itail}
if we take $t=k^{1/3}R^2$.

For $I_1$ we simply have
$$
 |I_1|\leq \int_{-2}^2
  |\zeta|^p
 \left|\tilde{r}_k(\zeta)
\right|
d\zeta \leq C ||  \tilde{r}_k||_{L^\infty(\Real)},
$$
with $C$
independent of $k$.
If $2k^{1/3}R\leq 1$ we have $I_2=0$ and the proof is
complete. We assume henceforth that $2k^{1/3}R> 1$
and let
$$
b(\zeta)=(1-\psi(\zeta))
\psi\left(\frac{\zeta}{k^{1/3}R}\right)
  \tilde{r}_k(\zeta),
$$
the support of which
lies in the compact set 
 $K=\{ \zeta\in\Real\ |\ 1\leq |\zeta| \leq 2k^{1/3}R \}$.
Then for $\zeta\in K$, by \eqref{etadeltasmall},
$$
|\partial_\zeta{\phi_a}(\tilde\delta,\zeta,\tilde\eta)|
=
|\zeta^2 +k^{2/3}(2\eta_0\eta-\delta)|\geq 
|\zeta^2|- k^{2/3}(2|\eta|+|\delta|)\geq 
|\zeta^2|- \frac{\xi_0^2}2\geq 
\frac12|\zeta^2|.
$$
Thus, since $\phi_a$ has no stationary points on $K$
we can use Lemma~\ref{nonstatphase}
with 
$a(\zeta)=b(\zeta)\zeta^p$ and $D$ an open bounded set
containing $K$. This gives
$$
  I_2 = 
  (ik)^{-n}
\sum_{\ell=0}^n 
\int_{K}\left(\frac{d^\ell}{d\zeta^\ell}b(\zeta)\zeta^p \right)
  \frac{u_{\ell,n}(\zeta)}{{\tilde\phi_\zeta(\tilde\delta,\zeta,\tilde\eta)}^{n}}
   e^{ik\phi_a(\tilde\delta,\zeta,\tilde\eta)}
d\zeta,
$$
where $u_{\ell,n} \in {\mathcal U}_{n-\ell}$, with
 ${\mathcal U}_{p}$ defined in (\ref{Udef}).
This expression is now estimated in the same way as
\eqref{Inonstat1} above.
Since $|\tilde{\delta}|\leq \xi_0^2/2\leq 1/2$
and $K$ satsfies the assumptions of Lemma~\ref{Uest}
we obtain a uniform bound for $u_{\ell,n}$ on $K$.
Then
\begin{align*}
 |I_2| &
 \leq  
  Ck^{-n}
\sum_{\ell=0}^n 
\sum_{j=0}^{\min(\ell,p)}
  ||b^{(\ell-j)}||_{L^\infty(\Real)}
\int_{1}^{2k^{1/3}R} \frac{d\zeta}{
|\zeta|^{2n-p+j}} 
 \leq  
  Ck^{-n}
  ||b||_{W^{n,\infty}(\Real)}
\int_{1}^{2k^{1/3}R} \frac{d\zeta}{
|\zeta|^{2n-p}} \\
&\leq
  Ck^{-n}
  ||b||_{W^{n,\infty}(\Real)}
  \max(1,(2k^{1/3}R)^{p-2n})2k^{1/3}R
\leq
  C
  ||b||_{W^{n,\infty}(\Real)}
  \max(k^{1/3-n},k^{(p-5n+1)/3}R^{p-2n})R_1\\
&\leq
  C \max(1,R_0^{p-2n},R_1^{p-2n})R_1
  ||b||_{W^{n,\infty}(\Real)} 
  =:\tilde{C}||b||_{W^{n,\infty}(\Real)},
\end{align*}
since $p+1\leq 5n$ and $k\geq 1$.
Moreover, as in the proof of Lemma~\ref{Itail}, 
by \eq{Wnineq1} and \eq{Wnineq2}, since $k^{1/3}R> 1/2$,
$$
|I_2|\leq \tilde{C}||b||_{W^{n,\infty}(\Real)}
\leq \tilde{C}C_n^22^n
||1-\psi||_{W^{n,\infty}(\Real)}
||\psi||_{W^{n,\infty}(\Real)}
||\tilde{r}_k||_{W^{n,\infty}(\Real)}
\leq C
||\tilde{r}_k||_{W^{n,\infty}(\Real)}.
$$
Together the estimates of $I_1$
and 
$I_2$ then prove the lemma.
We finally note that since $r\in W^{n,\infty}(\Real)$
the norm $||\tilde{r}_k||_{W^{n,\infty}(\Real)}$ is bounded
because of \eq{Wnineq2} and \eqref{etadeltasmall}.
\end{proof}

Finally, we show that the derivatives of the
Airy function are well approximated
by an oscillatory integral with a monomial 
factor and the Airy phase.
\begin{lem}\label{IAiry} 
Let 
$$
I = k^{1/3}
 \int \psi(\theta/R)\theta^p
e^{ik\phi_a(\delta,\theta,\eta)}
d\theta,
$$
where
$c_0\leq R$ with $c_0$ as in Lemma~\ref{phiest}
and $|\delta|\leq 1$.
If
$c_0(1+|\eta|^{1/2})\leq R$ and
$n\geq 1+p/2$, there is a constant $C_n$,
independent of $k$, $R$, $\delta$ and $\eta$,
such that
$$
 \left|I-
\frac{2\pi i^p}{k^{p/3}}\Ai^{(p)}(k^{2/3}(2\eta_0\eta-\delta))
\right|\leq  
C_nk^{1/3-n}.
$$
\end{lem} 
\begin{proof}
The Fourier transform of $\Ai^{(p)}$ in ${\mathcal S}'(\Real)$ is
$$
  \Fcal(\Ai^{(p)})(\zeta) = 
  \frac{1}{\sqrt{2\pi}}
(i\zeta)^p e^{i\frac{\zeta^3}{3}}.
$$
Therefore, using
Lemma \ref{lem:Fregularize},
and noting that ${\mathcal F}^{-1}(\cdot)(\rho)
={\mathcal F}(\cdot)(-\rho)$,
$$
\Ai^{(p)}(\rho) = 
{\mathcal F}^{-1}\left(
\frac{1}{\sqrt{2\pi}}
(i\zeta)^p e^{i\frac{\zeta^3}{3}}
\right)(\rho)=
\lim_{t\to\infty}
\frac{1}{2\pi}
 \int \psi(\zeta/t)(i\zeta)^p
e^{i\left(\frac{\zeta^3}{3}+\zeta\rho\right)}
d\zeta.
$$
After rescaling $\theta = k^{-1/3}\zeta$ we get
$$
 \int \psi(\zeta/t)(i\zeta)^p
e^{i\left(\frac{\zeta^3}{3}+\zeta\rho\right)}
d\zeta
= 
i^pk^{\frac{p+1}{3}}
\int \psi(\theta k^{1/3}/t) \theta^p
e^{ik\left(\frac{\theta^3}{3}+k^{-2/3}\theta\rho\right)}
d\theta.
$$
It follows that if 
\begin{equation}\label{rhodef}
\rho = k^{2/3}(2\eta_0\eta-\delta),
\end{equation}
then
\begin{align*}
\frac{2\pi}{(ik^{1/3})^p}\Ai^{(p)}(k^{2/3}(2\eta_0\eta-\delta)) &=
k^{1/3}
\lim_{t\to\infty}
 \int \psi(\theta k^{1/3}/t) \theta^p
e^{ik\left(\frac{\theta^3}{3}+k^{-2/3}\theta\rho\right)}
d\theta
\\
&=
k^{1/3}
\lim_{t\to\infty}
 \int \psi(\theta/t) \theta^p
e^{-ik\phi_a(\delta,\theta,\eta)}
d\theta \\
&=
k^{1/3}(-1)^{p}
\lim_{t\to\infty}
 \int \psi(-\theta/t) \theta^p
e^{-ik\phi_a(\delta,-\theta,\eta)}
d\theta.
\end{align*}
Since $\phi_a$ is odd and $\psi$ is even in $\theta$ we obtain
\begin{align*}
\frac{2\pi i^p}{k^{p/3}}\Ai^{(p)}(k^{2/3}(2\eta_0\eta-\delta))
&=
I+k^{1/3}
 \lim_{t\to\infty} \int (1-\psi(\theta/R))\psi(\theta/t)\theta^p
e^{ik\phi_a(\delta,\theta,\eta)}
d\theta,
\end{align*}
The result now follows from Lemma~\ref{Itail}, with $r\equiv 1$.
\end{proof}

\section{Solution estimates}\label{SolEst}

In Sections \ref{sec:Exactsols} and \ref{sec:GBapp}
it was shown that 
the partial Fourier
transform in $y$
for
both the exact solution and the
Gaussian beam approximation  can be written on the form
$$
  \hat{u}(x,\eta) = \hat{v}(x,\eta;k)\hat{A}(\eta),
$$
where $A$ is the amplitude function
and $\hat{v}$ for the two cases
are given in \eqref{vhatdef2} and \eqref{vgbhatdef}.
In this section we prove bounds of those $\hat{v}$
in terms of $k$
and $\eta$, which are valid
for all $\eta\in \Real$, $x\in [0,x_c]$ and $k\geq 1$.
We start with the Gaussian beam superposition case
and estimate ${\hat v}_{GB}$ as follows.

\begin{lem}\label{vgbest}
For $\hat{v}_{GB}$ defined in \eqref{vgbhatdef},
there is a constant $M$ such that 
\begin{equation}\label{vgbesteq}
   |\hat{v}_{GB}(\eta,x,k)| \leq M \left(1+\log\left(1+|\eta|^{1/2}\right)\right)k^{1/2},
\end{equation}
for all $\eta\in \Real$, $x\in [0,x_c]$ and $k\geq 1$.
\end{lem}
\begin{proof}
From \eqref{Itdef}
in
Section~\ref{GBsuper} 
we have
$$
I_t=
k^{1/3}
 \int\psi(\theta/t)
 r(\theta)
e^{ik\phi_g(\delta,\theta,\eta)}
d\theta,\qquad
r(\theta) = \frac{1}{2\pi}\sqrt{\frac{q(\xi_0)}{q(\xi_0+\theta)}
}.
$$
We note that $r\in 
W^{n,\infty}(\Real)$ for all $n$ by Lemma~\ref{qest}
and that $|\delta|=|x_c-x|\leq |x_c|=\xi_0^2\leq 1$.
We divide the integral into two parts.
Let $R=R(\eta)=c_0\left(1+|\eta|^{1/2}\right)$ 
with $c_0$ as in Lemma \ref{phiest} and define
$$
I_t  =
k^{1/3}
 \int
 \psi(\theta/R)\psi(\theta/t)r(\theta)
e^{ik\phi_g(\delta,\theta,\eta)}
d\theta+
k^{1/3}
 \int
 (1-\psi(\theta/R))\psi(\theta/t)r(\theta)
e^{ik\phi_g(\delta,\theta,\eta)}
d\theta 
  =: I_{1,t}+I_{2,t}.
$$
For $I_{1,t}$ we have,
again using Lemma~\ref{qest}, 
and that fact that $\Im\phi_g\geq 0$ by Lemma~\ref{phiest},
\begin{align*}
  |I_{1,t}| &\leq 
  k^{1/3}
  \sqrt{\frac{q_1}{q_0}}
  \int_{-2R(\eta)}^{2R(\eta)} \frac{e^{-k\Im\phi_g(x,\theta,\eta)}}{\sqrt{1+\theta^2}}d\theta \leq 
    k^{1/3}
  \sqrt{\frac{q_1}{q_0}}
\int_{-2R(\eta)}^{2R(\eta)} \frac{d\theta}{\sqrt{1+\theta^2}}
\\
   &\leq C    k^{1/3}\log(R(\eta))
   \leq C'    k^{1/3}\log(1+|\eta|^{1/2}).
\end{align*}
For $I_{2,t}$ we use Lemma~\ref{Itail} with $p=0$, 
which says that for any $n\geq 1$
and $t>R(\eta)$,
$$
  |I_{2,t}|\leq 
  C_n k^{1/3-n}
  ||r||_{W^{n,\infty}(\Real)}
  \leq C_n'k^{1/3-n},
$$
where $C_n'$ is independent of $k$, $R(\eta)$, $\delta$, $\eta$ and $t$.
By \eqref{pgbdef} we have
$|P_{GB}|=2(\pi \xi_0)^{1/2}\leq 2\sqrt{\pi}$ 
for all $\eta$ and $k$.
Then 
we  get
\begin{align*}
|\hat{v}_{GB}(\eta,x,k)|&\leq
\lim_{t\to\infty}
k^{1/6}|P_{GB}(x,\eta)|(|I_{1,t}|+|I_{2,t}|)
\leq 
2\sqrt{\pi}k^{1/6}\Bigl[C'k^{1/3}\log(1+|\eta|^{1/2})+
C_n'k^{1/3-n}\Bigr],
\end{align*}
and upon taking $n=1$
the Lemma follows with $M=2\sqrt{\pi}\max(C',C'_1)$.
\end{proof}


Next, for the exact solution, we estimate
$\hat{v}$.

\begin{lem}\label{vexactest2}
For $\hat{v}$ defined in \eqref{vhatdef2}
there is a constant $M$ such that
$$
   |\hat{v}(\eta,x,k)| \leq M k^{1/6},
$$
for all $\eta\in \Real$, $x\in [0,x_c]$ and $k\geq 1$.
\end{lem}
\begin{proof}
From \eqref{vhatdef2} we have
$$
{\hat v}(\eta,x,k) = k^{1/6} P(k,\eta)\Ai(k^{\frac23}(x-X)),
\qquad P(k,\eta) = 
\frac{\bar\alpha k^{-1/6}}{\Ai(\alpha k^{2/3}X)},\qquad
X(\eta) = 1-(\eta_0+\eta)^2.
$$
Then, using \eq{eq:Aisest} and \eq{eq:Aiasest} in Lemma~\ref{lem:Aiprop},
\begin{align*}
   |\hat{v}(\eta,x,k)| &= 
   \frac{|\Ai(k^{\frac23}(x-X))|}{|\Ai({\alpha}k^{\frac23}X)|}
   \leq 
   C\frac{(1+k^{\frac23}|x-X|)^{-1/4}}{(1+k^{\frac23}|X|)^{-1/4}}
\leq 
   C\left(
   \frac{1+k^{\frac23}(|x|+|x-X|)}{1+k^{\frac23}|x-X|}
   \right)^{1/4}
 \\ &=
   C\left(
   1+
   \frac{k^{\frac23}|x|}{1+k^{\frac23}|x-X|}
   \right)^{1/4}
   \leq
   C\left(
   1+
|x_c|k^{\frac23}\right)^{1/4}\leq C(1+k^{1/6}).
\end{align*}
The result follows as $k\geq 1$.
(The estimate is sharp for $x=X$.)
\end{proof}

\section{Proof of the Main Result}

In this section we prove the main result Theorem \ref{mainresult}
estimating the $L^{\infty}$ error 
between the exact solution and the Gaussian beam solution.
To estimate the  difference between $u_{GB}(x,y)$ and 
the exact solution
$u(x,y)$, it is enough to control the $L^1$ norm of the difference between their scaled Fourier transforms
since
\begin{equation}\label{L1est1}
  ||u_{GB}(x,\cdot)-u(x,\cdot)||_{L^\infty}
  \leq\sqrt{\frac{k}{2\pi}}
||{\hat u}_{GB}(x,\cdot)-{\hat u}(x,\cdot)||_{L^1}.
\end{equation}
We will use this strategy.
From (\ref{vgbhatdef}) and (\ref{vhatdef2})
we get
$$
{\hat u}_{GB}(x,\eta_0+\eta) -
{\hat u}(x,\eta_0+\eta) = 
\bigl({\hat v}_{GB}(\eta,x,k)-
{\hat v}(\eta,x,k)\bigr) {\hat{A}}(\eta).
$$
We
divide the expression into two parts, one for $|\eta|$
smaller than $O(k^{-2/3})$ and one for $|\eta|$ larger than $O(k^{-2/3})$.
Thus, for $c$ to be determined below, we let
\begin{align}\label{L1est2}
  \sqrt{\frac{k}{2\pi}}
||{\hat u}_{GB}(x,\cdot)-{\hat u}(x,\cdot)||_{L^1}
  &\leq\sqrt{\frac{k}{2\pi}}
  \int_{|\eta|\leq ck^{-2/3}}
  |{\hat v}_{GB}(x,\eta)-{\hat v}(x,\eta)||{\hat{A}}(\eta)|d\eta\\
  &\ \ \ +\sqrt{\frac{k}{2\pi}}\int_{|\eta|\geq ck^{-2/3}}
  |{\hat v}_{GB}(x,\eta)-{\hat v}(x,\eta)||{\hat{A}}(\eta)|d\eta
  =: E_1+E_2.\nonumber
\end{align}
For the large values of $|\eta|$ we can immediately get a bound
of $O(k^{-1})$ by using the fact that $\hat{A}$ has very rapid
decay, being the Fourier transform
of $A\in\Scal$. This is used in the following lemma.

\begin{lem}\label{etaest1}
Suppose $f(\cdot,k)\in L^{1}_{\rm loc}(\Real)$
for each $k\geq 1$
and $A\in {\mathcal S}(\R)$. 
Let $c,\beta>0$ be given.
If
there exist $r\in\R$ and
$q,M\in\R_+$ such that, 
\begin{equation}\label{fcond1}
 |f(\eta,k)|\leq M(1+|k\eta|^q)k^r,\qquad \forall  k\geq 1, \quad
 |k\eta|\geq ck^\beta,
 \end{equation}
 then, for each $p\geq 0$, there exists a constant $C_p$,
 independent of $k$, but dependent on $A,M,q,r,c$,
 such that
 $$
    k^{1/2}\int_{|k\eta|\geq ck^\beta} \left|f(\eta,k)\hat{A}(\eta)\right|d\eta
    \leq C_p k^{-p}.
 $$
\end{lem}

\begin{proof}
Since $A\in\mathcal{S}(\R)$
for all $\ell\geq 0$
there exists $c_\ell$ such that $\vert \mathcal{F}(A)(\eta)\vert\leq c_\ell/(1+\vert \eta\vert^{\ell})$.
Hence,
$$
|\hat{A}(\eta)|=
|\sqrt{k}\Fcal(A)(k\eta)|
\leq
\frac{c_\ell\sqrt{k}}{1+|k\eta|^\ell},
\qquad \forall \eta.
$$
Then 
for $\ell>q+1$, with $\xi=k\eta$,
\begin{align*}
    k^{1/2}\int_{|k\eta|\geq ck^\beta} \left|f(\eta,k)\hat{A}(\eta)\right|
    d\eta
&\leq
c_\ell k
\int_{|k\eta|\geq ck^\beta} \frac{|f(\eta,k)|}{1+|k\eta|^\ell} d\eta\leq
c_\ell M k^{r+1}
\int_{|k\eta|\geq ck^\beta}\frac{1+|k\eta|^q}{1+|k\eta|^\ell}d\eta
\\
&=
c_\ell M k^{r}
\int_{|\xi|\geq ck^\beta}\frac{1+|\xi|^q}{1+|\xi|^\ell}d\xi
\leq 
c_\ell M k^{r}
\int_{|\xi|\geq ck^\beta}\frac{1+|\xi|^q}{|\xi|^\ell}d\xi
\\
& =
2Mc_\ell \left(
\frac{c^{1-\ell}}{\ell-1}k^{r+\beta(1-\ell)}
+
\frac{c^{1-\ell+q}}{\ell-1-q}  k^{r+\beta(1-\ell+q)}
\right).
\end{align*}
For the given $p$ we now 
take $\ell=\ell_p:=\max((r+p)/\beta+q+1,q+2)$.
Then $r+\beta(1-\ell+q)\leq -p$ and
\begin{align*}
    k^{1/2}\int_{|k\eta|\geq ck^\beta} \left|f(\eta,k)\hat{A}(\eta)\right|d\eta
&\leq
2Mc_{\ell_p} \left(
\frac{c^{1-\ell_p}}{\ell_p-1}k^{-p-\beta q}
+
\frac{c^{1-\ell_p+q}}{\ell_p-1-q}  k^{-p}
\right)\\
&\leq 
\underbrace{
2Mc_{\ell_p} \left(
\frac{c^{1-\ell_p}}{\ell_p-1}
+
\frac{c^{1-\ell_p+q}}{\ell_p-1-q}
\right)}_{=:C_p}  k^{-p},
\end{align*}
which is the desired estimate. 
\end{proof}
In our case we let $f=(\hat{v}_{GB}-{\hat v})/\sqrt{2\pi}$ for fixed $x$.
Then it follows from Lemma~\ref{vgbest} 
and Lemma \ref{vexactest2}
that $f$ satisfies \eq{fcond1}, as for $k\geq 1$,
$$
|\hat{v}_{GB}(\eta,x,k)-\hat{v}(\eta,x,k)|
\leq Mk^{1/6}+ M(1+\log(1+|\eta|^{1/2}))k^{1/2}
\leq M'(1+|\eta|^{1/2})k^{1/2}
\leq M'(1+|k\eta|^{1/2})k^{1/2},
$$
where $M'$ is in fact independent also of $x\in [0, x_c]$.
Lemma~\ref{etaest1} with
$\beta=1/3$, $q=1/2$, $r=1/2$, $c=\xi_0^2/4$ and $p=1$ now shows that
\begin{equation}\label{E2est}
  |E_2|\leq C k^{-1}.
\end{equation}
The choice of $c$ will be motivated below in the next step.

To estimate $E_1$ we will use more precise estimates of 
$|\hat{v}_{GB}(\eta,x,k)-\hat{v}(\eta,x,k)|$ for small $|\eta|$ and the
following lemma.
\begin{lem}\label{etaest2}
Suppose $f(\cdot,k)\in L^{1}_{\rm loc}(\Real)$
for each $k\geq 1$
and $A\in {\mathcal S}(\R)$. 
Let $c,\beta>0$ be given.
If
there exist $r\in\R$ and
$q,M\in\R_+$ such that, 
\begin{equation}\label{fcond2}
 |f(\eta,k)|\leq M(1+|k\eta|^q)k^r,\qquad \forall  k\geq 1, \quad
 |k\eta|\leq ck^\beta,
 \end{equation}
 then there exists a constant $C$,
 independent of $k$, but dependent on $A,M,q,r,c$,
 such that
 $$
    k^{1/2}\int_{|k\eta|\leq ck^\beta} \left|f(\eta,k)\hat{A}(\eta)\right|
    d\eta\leq C k^{r}.
 $$
\end{lem}

\begin{proof}
As in the proof of Lemma~\ref{etaest1}, when $\ell>q+1$ we 
get
\begin{align*}
    k^{1/2}\int_{|k\eta|\leq ck^\beta} \left|f(\eta,k)\hat{A}(\eta)\right|d\eta
&\leq
c_\ell M k^{r}
\int_{|\xi|\leq ck^\beta}\frac{1+|\xi|^q}{1+|\xi|^\ell}d\xi\leq
\underbrace{c_\ell M 
\int\frac{1+|\xi|^q}{1+|\xi|^\ell}d\xi}_{=:C} k^{r}.
\end{align*}
This proves the lemma.
\end{proof}
As for $E_2$ we apply this lemma with 
$f=(\hat{v}_{GB}-{\hat v})/\sqrt{2\pi}$
and to get the bound \eq{fcond2} we need
to estimate the difference between
${\hat v}_{GB}$ and $\hat{v}$ when $|k\eta|\leq ck^{1/3}$.
This estimate is the main part of the proof. 

To examine ${\hat v}_{GB}-\hat{v}$ more carefully
we first recall the expressions:
\begin{align*}
{\hat v}_{GB}(\eta,x,k) &=  k^{1/6} P_{GB}(k,\eta) I(\eta,x,k),\\
\hat{v}(\eta,x,k) 
&=k^{1/6} P(k,\eta)\Ai(k^{\frac23}(x-X(\eta))),
\end{align*}
where
$$
I(\eta,x,k)=
\lim_{t\to\infty}
 \frac{k^{1/3}\sqrt{q(\xi_0)}}{2\pi}
 \int
 \psi(\theta/t)
\frac{e^{ik\phi_g(x,\theta,\eta)
}
}
{\sqrt{q(\xi_0+\theta)}} 
d\theta\,\qquad
P_{GB}(k,\eta)=
2
(\pi \xi_0)^{1/2}
e^{-i\pi/4}
e^{\frac23ik\xi_0^3-2ik\eta_0\xi_0\eta},
$$
and
\be\label{Xdef}
X(\eta) := 1-(\eta+\eta_0)^2,
\qquad P(k,\eta) = 
\frac{\bar\alpha k^{-1/6}}{\Ai(\alpha k^{2/3}X)}.
\ee
We
divide the difference $\hat{v}_{GB}(\eta,x,k)-\hat{v}(\eta,x,k)$ 
into three parts
\begin{align*}
\hat{v}(\eta,x,k)-\hat{v}_{GB}(\eta,x,k)&=
k^{1/6} P_{GB}(k,\eta)[\Ai(k^{\frac23}(x-X))-\Ai(k^{\frac23}(x-X-\eta^2))]
\\
&\ \ \ +k^{1/6} [P(k,\eta)-P_{GB}(k,\eta)]\Ai(k^{\frac23}(x-X))
\\
&\ \ \ +k^{1/6} P_{GB}(k,\eta)[\Ai(k^{\frac23}(x-X-\eta^2))-I(\eta,x,k)]
\\
&=: R_1 + R_2 + R_3.
\end{align*}
In three Lemmas below we show that when $x=x_c$ and $\eta\leq \xi_0^2k^{-2/3}/4$, there is a constant $M$ such that
$$
  |R_1| \leq M|k\eta|^2k^{-1},\qquad
  |R_2| \leq M(1+|k\eta|^2)k^{-5/6},\qquad
  |R_3| \leq M(1+|k\eta|^2)k^{-5/6}.
$$
It follows that
$$
|\hat{v}(\eta,x_c,k)-\hat{v}_{GB}(\eta,x_c,k)|\leq
  |R_1+R_2+R_3| \leq 3M(1+|k\eta|^2)k^{-5/6}, \qquad
  \text{when $|k\eta|\leq \frac{\xi_0^2}{4}k^{1/3}$ and $k\geq 1$},
$$
Then applying 
Lemma~\ref{etaest2} with
$\beta=1/3$, $q=2$, $r=-5/6$ and $c = \xi_0^2/4$ shows that
$$
  |E_1|\leq C k^{-5/6},
$$
when $x=x_c$. Together with \eqref{L1est1}, \eqref{L1est2} and \eqref{E2est}
this proves Theorem~\ref{mainresult}.

Note that the estimates of  $R_1$ and $R_2$ above are shown to be
valid for all  $x\in[0,x_c]$, while the $R_3$ estimate is considered,
in this paper, only for $x=x_c$.
Furtheremore, 
note that $R_2$ and $R_3$ exhibit the same loss of decay through
the term $k^{-5/6}$. In $R_2$ this comes from the estimate \eq{diffX}
and $R_3$ has $k^{1/6}$ embedded in \eq{R3est}.

We now turn to proving the lemmas about $R_j$.

\subsection{Estimate of $R_1$}

\begin{lem}
There is a constant $M$ independent of
$\eta$ and $k\geq 1$, such that
$$
  |R_1| \leq Mk\eta^2, \qquad \text{when}\ |\eta| \leq 1.
$$
\end{lem}
\begin{proof}
Since $|P_{GB}|\leq 2(\pi\xi_0)^{1/2}$
we have
$$
 |R_1| \leq 2\sqrt{\pi}
k^{1/6}|\Ai(k^{\frac23}(x-X))-\Ai(k^{\frac23}(x-X-\eta^2))|.
$$
Moreover, 
from \eq{eq:Aiprimesest} in Lemma \ref{lem:Aiprop}
 we get
\begin{align*}
  |\Ai(k^{\frac23}(x-X))-\Ai(k^{\frac23}
  (x-X-\eta^2))|
  &=k^{\frac23}\left|\int_0^{\eta^2}
\Ai'(k^{\frac23}(x-X-s))ds\right|
\\
&
\leq 
C_3 k^{\frac23} \eta^2 \max_{0\leq s\leq \eta^2}
(1+|k^{\frac23}(x-X-s)|)^{1/4}.
\end{align*}
Then, since 
$$
|X| \leq (|\eta|+|\eta_0|)^2-1 \leq 3,\qquad |x|\leq 1, \qquad
s\leq\eta^2\leq 1,
$$
we obtain
$$
|\Ai(k^{\frac23}(x-X))-\Ai(k^{\frac23}(x-X-\eta^2))|
\leq 
C_3 k^{\frac23} \eta^2 
(1+5k^{\frac23})^{1/4}
\leq C_36^{1/4}k^{5/6}\eta^2.
$$
It follows that $|R_1|\leq M k\eta^2$ where $M=2\sqrt{\pi}6^{1/4}C_3$
with $C_3$ being the constant in \eq{eq:Aiprimesest}.
\end{proof}

\subsection{Estimate of $R_2$}

\begin{lem}\label{R2lemma}
There is a constant $M$ dependent on $x$, but independent of
$\eta$ and $k\geq 1$, such that
$$
  |R_2| \leq M(1+k^2\eta^2)k^r, \qquad \text{when}\ |\eta| \leq \frac{\xi_0^2}{4}k^{-2/3}, \qquad
  r = 
  \begin{cases}
-5/6, & x=x_c,\\
-1, & 0\leq x<x_c.
\end{cases}
$$
\end{lem}
\begin{proof}
We start by proving two estimates of $X(\eta)$.
We use the inequalities $1-x\leq\sqrt{1-x}\leq 1-x/2$ which hold for $x\in[0,1]$.
The definition \eq{Xdef} 
together with the assumption on $\eta$ and the fact that $k\geq 1$,
then gives
$$
X(\eta) = 1-(\eta+\eta_0)^2\geq 
1-\left(\frac{\xi_0^2}{4}+\sqrt{1-\xi_0^2}\right)^2\geq 
1-\left(1-\frac{\xi_0^2}{4}\right)^2\geq \frac{\xi_0^2}{4}=:X_0>0.
$$
Moreover, since $x_c=\xi_0^2$,
$$
|X(\eta)-x_c| 
=|\eta| |2\eta_0-\eta|
\leq |\eta| \left|2\sqrt{1-\xi_0^2}+\frac{\xi_0^2}{4}\right|
\leq |\eta| \left|2-\xi_0^2+\frac{\xi_0^2}{4}\right|\leq 2|\eta|.
$$
Clearly, we also have $X(\eta)\leq 1$, and therefore, in summary,
\be\label{Xdiff}
 0<X_0 \leq X(\eta)\leq 1, \qquad
|X(\eta)-x_c|\leq 2|\eta|.
\ee

Next, we rewrite $P_{GB}$, adopting the 
the defintion
\be\label{Aiwuse}
\widetilde{\Ai}(z)=\frac1{2\sqrt\pi} z^{-\frac14}e^{-\frac23 z^{\frac32}},
\ee
from Lemma \ref{lem:Aiprop}. Then for $x\in\Real$,
$$
\widetilde{\Ai}(\alpha k^{2/3}x)=
\frac{k^{-1/6}}{2\sqrt\pi}e^{-i\pi/12} x^{-\frac14}e^{-\frac23 ikx^{3/2}}
\quad\Rightarrow\quad
\widetilde\Ai(\alpha k^{2/3}X)=
\widetilde\Ai(\alpha k^{2/3}x_c)\left(\frac{x_c}{X}\right)^{1/4}
e^{-\frac23ik (X^{3/2}-x_c^{3/2})},
$$
and
$$
P_{GB} = \frac{\bar\alpha k^{-1/6}}{\widetilde{\Ai}(\alpha k^{2/3}x_c)}
e^{-2ik\eta\eta_0\sqrt{x_c}}
=
\frac{\bar\alpha k^{-1/6}}{\widetilde{\Ai}(\alpha k^{2/3}X)}
\left(\frac{x_c}{X}\right)^{1/4}
e^{-\frac23ik (X^{3/2}-x_c^{3/2})-2ik\eta\eta_0\sqrt{x_c}}
=:
\frac{\bar\alpha k^{-1/6}g(\eta)}{\widetilde{\Ai}(\alpha k^{2/3}X)}.
$$
We get
\begin{align*}
 k^{1/6}|P-P_{GB}| &
=
\left|\frac{1}{\Ai(\alpha k^{2/3}X)}-
\frac{g(\eta)}{\widetilde{\Ai}(\alpha k^{2/3}X)}
\right|
\leq
\left|\frac{1}{\Ai(\alpha k^{2/3}X)}-\frac{1}{\widetilde\Ai(\alpha k^{2/3}X)}\right|
+\left|\frac{1-g(\eta)}{\widetilde\Ai(\alpha k^{2/3}X)}\right|
\\
&=
\frac{1}{|\widetilde\Ai(\alpha k^{2/3}X)|}
\left(
\left|\frac{\widetilde\Ai(\alpha k^{2/3}X)-\Ai(\alpha k^{2/3}X)}{\Ai(\alpha k^{2/3}X)}\right|
+|1-g(\eta)|\right).
\end{align*}
We can then estimate $R_2$ as
\begin{align}\label{R2split}
 |R_2| &\leq
 \frac{|\Ai(k^{\frac23}(x-X))|}{|\widetilde\Ai(\alpha k^{2/3}X)|}
\left(
\left|\frac{\widetilde\Ai(\alpha k^{2/3}X)-\Ai(\alpha k^{2/3}X)}{\Ai(\alpha k^{2/3}X)}\right|
+|1-g(\eta)|\right).
\end{align}
We will now study the different parts of this expression separately.

\begin{itemize}

\item {\bf Estimate of $\left|\frac{\widetilde\Ai(\alpha k^{2/3}X)-\Ai(\alpha k^{2/3}X)}{\Ai(\alpha k^{2/3}X)}\right|$.}

\noindent This is given directly by
\eq{eq:Aidiff2} in
Lemma \ref{lem:Aiprop} with $s_0=X_0$,
as then $k^{2/3}X(\eta)\geq k^{2/3}X_0\geq s_0$. We get
\begin{align*}
\left|\frac{\widetilde\Ai(\alpha k^{2/3}X)-\Ai(\alpha k^{2/3}X)}{\Ai(\alpha k^{2/3}X)}\right|
&\leq C_1 |k^{2/3}X|^{-3/2}\leq C_1X_0^{-3/2} k^{-1}
=: D_1 k^{-1},
\end{align*}
where $C_1$ is the constant in \eq{eq:Aidiff2}.

\item {\bf Estimate of $|1-g(\eta)|$.}

\noindent Using Taylor's formula for $x\mapsto x^{3/2}$
around $x=X(0)
=1-\eta_0^2=\xi_0^2=x_c$, we compute
\begin{align*}
  X(\eta)^{3/2}&=X(0)^{3/2}+\frac32X(0)^{1/2}(X(\eta)-X(0))
  +R (X(\eta)-X(0))^2,\\
  &=x_c^{3/2}+\frac32\sqrt{x_c}(2\eta\eta_0-\eta^2)
  +R (X(\eta)-X(0))^2,\qquad |R|\leq \sup_{\xi\geq X_0}\frac38 \xi^{-1/2}
  =\frac3{8 \sqrt{X_0}}.
\end{align*}
Therefore, 
$$
\frac23(X^{3/2}-x_c^{3/2})
+2\eta\eta_0\sqrt{x_c}
=
-\sqrt{x_c}\eta^2
  +\frac23 R (X(\eta)-x_c)^2,
$$
and consequently, by \eq{Xdiff},
$$
\left|
\frac23(X^{3/2}-x_c^{3/2})
+2\eta\eta_0\sqrt{x_c}
\right| 
\leq \left(
|\eta|^2+
\frac{|X(\eta)-x_c|^2}{4 \sqrt{X_0}}
\right)
\leq \left(
|\eta|^2+
\frac{|\eta|^2}{\sqrt{X_0}}
\right)=:D_2 \eta^2.
$$
This gives us
\begin{align*}
\left|
1-g(\eta)
\right|
&\leq
\left|
1-
e^{-ik\left(
\frac23(X^{3/2}-x_c^{3/2})
+2\eta\eta_0\sqrt{x_c}\right)}
\right|
+
\left|
\left(\frac{x_c}{X(\eta)}\right)^{1/4}-1
\right|
\\
&\leq 
k\left|
\frac23(X^{3/2}-x_c^{3/2})
+2\eta\eta_0\sqrt{x_c}
\right|
+
\left|
\left(1+\frac{2\eta}{X_0}\right)^{1/4}-1
\right|\\
&\leq
D_2k\eta^2 + \frac{1}{2X_0}\eta
\leq D_3(1+k^2\eta^2)k^{-1},
\end{align*}
for $D_3=\max(D_2,(2X_0)^{-1})$.
\item {\bf Estimate of $\frac{|\Ai(k^{\frac23}(x-X))|}{|{\widetilde\Ai}(\alpha k^{2/3}X)|}$.}

\noindent 
We divide this into three subcases.
Suppose first that 
$$|x-x_c|> |X(\eta)-x_c|.$$ 
Then by \eq{Xdiff},
\begin{align*}
k^{2/3}|x-X(\eta)|&\geq k^{2/3}|x-x_c|- k^{2/3}|X(\eta)-x_c|
\geq k^{2/3}|x-x_c|-2k^{2/3}|\eta|\\
&\geq k^{2/3}|x-x_c|-\frac{\xi_0^2}{2}.
\end{align*}
By \eq{eq:Aisest} in Lemma \ref{lem:Aiprop} and \eq{Aiwuse},
\begin{align*}
\frac{|\Ai(k^{\frac23}(x-X))|}{|{\widetilde\Ai}(\alpha k^{2/3}X)|}
&\leq
2\sqrt{\pi}k^{1/6}|X|^{1/4}C_2(1+|k^{\frac23}(x-X)|)^{-1/4}
\\
&\leq 2\sqrt{\pi}k^{1/6}C_2(1-\xi^2_0/2+k^{\frac23}|x-x_c|)^{-1/4}\leq 2\sqrt{\pi}C_2|x-x_c|^{-1/4},
\end{align*}
where $C_2$ is the constant in \eq{eq:Aisest}.
On the other hand, if
$$|X(\eta)-x_c|\geq |x-x_c|>0,$$ 
then by \eq{Xdiff},
\begin{align*}
k^{2/3}\leq \frac{\xi_0^2}{4|\eta|}
\leq \frac{\xi_0^2}{2|X(\eta)-x_c|}
\leq \frac{\xi_0^2}{2|x-x_c|},
\end{align*}
and we obtain the same estimate as above, via
\begin{align*}
\frac{|\Ai(k^{\frac23}(x-X))|}{|{\widetilde\Ai}(\alpha k^{2/3}X)|}
&\leq
2\sqrt{\pi}k^{1/6}|X|^{1/4}C_2(1+|k^{\frac23}(x-X)|)^{-1/4}
\leq 2\sqrt{\pi}k^{1/6}C_2
\\
&\leq
2\sqrt{\pi}C_2\left(\frac{\xi_0^2}{2|x-x_c|}\right)^{1/4}
\leq 2\sqrt{\pi}C_2|x-x_c|^{-1/4}.
\end{align*}
Finally, when $x=x_c$ (caustic case) we can not get better than
\begin{align*}
\frac{|\Ai(k^{\frac23}(x-X))|}{|{\widetilde\Ai}(\alpha k^{2/3}X)|}
&\leq
2\sqrt{\pi}k^{1/6}|X|^{1/4}C_2(1+|k^{\frac23}(x-X)|)^{-1/4}
\leq
2\sqrt{\pi}C_2k^{1/6}.
\end{align*}
In summary, we have with $D_4 =2\sqrt{\pi}C_2$,
\begin{equation}\label{diffX}
\frac{|\Ai(k^{\frac23}(x-X))|}{|{\widetilde\Ai}(\alpha k^{2/3}X)|}
\leq
D_4
\begin{cases}
k^{1/6}, & x=x_c,\\
|x-x_c|^{-1/4}, & x<x_c.
\end{cases}
\end{equation}
\end{itemize}
We can now put the estimates together and apply them to $R_2$ in \eq{R2split}.
We get
$$
 |R_2|\leq 
 \left(D_1k^{-1} + D_3(1+k^2\eta^2)k^{-1}
 \right)
 D_4
 \begin{cases}
k^{1/6}, & x=x_c,\\
|x-x_c|^{-1/4}, & x<x_c,
\end{cases}\ \ 
\leq M(1+k^2\eta^2) k^{r},
$$
where 
$$
M=\max(D_1,D_3)D_4
 \begin{cases}
1, & x=x_c,\\
|x-x_c|^{-1/4}, & x<x_c.
\end{cases}
$$
This proves the lemma.
\end{proof}

\subsection{Estimate of $R_3$ at the caustic.}

This is the main estimate. Here we assume that $x=x_c$.
\begin{lem}\label{R3lemma}
For $x=x_c$
there is a constant $M$ independent of
$\eta$ and $k\geq 1$, such that
$$
  |R_3| \leq M(1+k^2\eta^2)k^{-5/6}, \qquad \text{when}\ |\eta| \leq \frac{\xi_0^2}{4}k^{-2/3}.
$$
\end{lem}

\begin{proof}
We consider 
 $\rho = 2k^{2/3}\eta_0\eta$, which amounts to taking $\delta=0$ in \eq{rhodef}.
By the assumption on $\eta$ and $k$ it is bounded as
\begin{equation}\label{rhobound}
  |\rho| \leq 2k^{2/3}|\eta| \leq \xi_0^2/2\leq \frac12.
\end{equation}
Moreover, since $k^{2/3}(x_c-X-\eta^2)=\rho$ and 
as before, $|P_{GB}|\leq 2(\pi\xi_0)^{1/2}$, we get
\begin{equation}\label{R3est}
|R_3|
= k^{1/6}|P_{GB}(k,\eta)|\cdot |\Ai(k^{\frac23}(x_c-X-\eta^2))-I(\eta,x_c,k)|
\leq 
Ck^{1/6}|\Ai(\rho)-I(\eta,x_c,k)|.
\end{equation}
Hence, we need to estimate $|\Ai(\rho)-I|$.

Let 
\be\label{rdef}
r(\theta) = \frac{1}{2\pi}\sqrt{\frac{q(\xi_0)}{q(\xi_0+\theta)}}.
\ee
As in the proof of Lemma~\ref{vgbest} we then
use the fact that
$I=\lim_{t\to\infty} I_t$
where $I_t$ is defined and divided as
\begin{align*}
I_t&=
k^{1/3}
 \int\psi(\theta/t)
 r(\theta)
e^{ik\phi_g(0,\theta,\eta)}
d\theta
=
k^{1/3}
 \int
 \psi(\theta/R)\psi(\theta/t)r(\theta)
e^{ik\phi_g(0,\theta,\eta)}
d\theta
\\
&\ \ +
k^{1/3}
 \int
 (1-\psi(\theta/R))\psi(\theta/t)r(\theta)
e^{ik\phi_g(0,\theta,\eta)}
d\theta 
  =: I_{\rm main}+I_{t,\rm tail}.
\end{align*}
 With $c_0$ as in Lemma \ref{phiest} we choose here 
$R=3c_0/2$, independent of $\eta$, which implies that
for all $\eta$ which we consider,
\begin{equation}\label{Rbound}
  c_0(1+|\eta|^{1/2})\leq 
  c_0(1+\xi_0/2)\leq \frac32 c_0 =R.
\end{equation}
Moreover, we take
$t\geq 2R=3c_0$ such that $\psi(\theta/R)\psi(\theta/t)=\psi(\theta/R)$.
To analyze $I_{\rm main}$ we then
first note that it can be written as
$$
I_{\rm main}=
k^{1/3}
 \int
 \psi(\theta/R)\tilde{r}(\theta)
e^{ik\phi_a(0,\theta,\eta)}
d\theta,
\qquad \tilde{r}(\theta) = r(\theta)e^{ik\frac12m_{11}(\xi_0+\theta)\theta^4}.
$$
We next expand $\tilde{r}$ in terms of $\theta$, first using
the Taylor expansion of $\exp(iz)$,
$$
e^{iz} = 1 + iz +\frac{(iz)^2}{2} + (iz)^3 Z(z),
\qquad  Z(z) = \frac{1}{2i}\int_0^1 e^{isz}(1-s)^2ds.
$$
This gives
\begin{align*}
 \tilde{r}(\theta)
 &=r(\theta)\left(1+ \frac{ik}2 m_{11}(\xi_0+\theta)\theta^4
 +\frac{(ik)^2}8 m_{11}(\xi_0+\theta)^2\theta^8
 \right)\\
 &\ \ \ 
 +\frac{(ik)^3}8r(\theta)Z\left(\frac{k}2 m_{11}(\xi_0+\theta)\theta^4\right)
 m_{11}(\xi_0+\theta)^3\theta^{12}.
\end{align*}
Furthermore, let
$$
  v_\ell(\theta):= r(\theta)m_{11}(\xi_0+\theta)^\ell,
$$
and Taylor expand these functions as
$$
v_\ell(\theta) = \sum_{j=0}^p \frac{v_\ell^{(j)}(0)\theta^j}{j!}+
V_{p,\ell}(\theta)\theta^{p+1},\qquad
V_{p,\ell}(\theta)=
  \frac{1}{p!}
\int_0^1 v_\ell^{(p+1)}(t\theta)(1-t)^p dt.
$$
Then
\begin{align*}
 \tilde{r}(\theta)
 &=v_0(0)+v_0'\theta +\frac12v_0''(0)\theta^2  + V_{2,0}(\theta)\theta^3 \\
 &\ \ +\frac{ik}2 v_1(0)\theta^4
 +\frac{ik}2 v_1'(0)\theta^5
 +\frac{ik}2 V_{1,1}(\theta)\theta^6\\
 &\ \ +\frac{(ik)^2}8 v_2(0)\theta^8+\frac{(ik)^2}8V_{0,2}(\theta)\theta^9
 +\frac{(ik)^3}8v_3(\theta)Z\left(k\frac12 m_{11}(\xi_0+\theta)\theta^4\right)\theta^{12}.
\end{align*}
From this expansion of $\tilde{r}$ we now get a corresponding expansion of
$I_{\rm main}$,
\begin{align}\label{Idivide}
I_{\rm main}  &=
I_S +
I_{V_{2,0}}+\frac{1}2I_{V_{1,1}}+\frac{1}8I_{V_{0,2}}
+I_Z,\\
I_{S}  &=
v_0(0)I_0+v_0'(0)I_1 +\frac12v_0''(0)I_2
+v_1(0)\frac{ik}2I_4
 +\frac{ik}2 v_1'(0)I_5
+\frac{(ik)^2}8 v_2(0)I_8\nonumber,
\end{align}
where
$$
I_p = k^{1/3}
 \int
 \psi(\theta/R)\theta^p
e^{ik\phi_a(0,\theta,\eta)}
d\theta,\qquad
I_{V_{p,\ell}} =
(ik)^\ell
 k^{1/3}
 \int
 \psi(\theta/R)V_{p,\ell}(\theta)\theta^{3\ell+3}
e^{ik\phi_a(0,\theta,\eta)}
d\theta,
$$
and
$$
I_Z = 
\frac{(ik)^3}{8}
 k^{1/3}
 \int
 \psi(\theta/R)w(\theta,k)\theta^{12}
e^{ik\phi_a(0,\theta,\eta)}
d\theta,
$$
with
$$
 w(\theta,k) = v_3(\theta)z(\theta,k),
\qquad
 z(\theta,k) = Z\left(\frac{k}{2}m_{11}(\xi_0+\theta)\theta^4\right).
$$

We will next show that the last four terms in \eq{Idivide} are 
at most of size $O(1/k)$.
To see this, we note that
by Lemma~\ref{qest}
and Lemma~\ref{m11est}, both $r$ and
$m_{11}(\xi_0+\cdot)$ belong to $W^{n,\infty}(\Real)$
for all $n$, $\xi_0$,
and their $W^{n,\infty}$-norms are bounded independent of $\xi_0$.
By \eq{Wnineq1} the same is true for $v_\ell$, for all $\ell$.
Therefore, by \eq{Wnineq2},
\begin{align*}
  ||V_{p,\ell}||_{W^{n,\infty}(\Real)} &\leq 
  \frac{1}{p!}
\int_0^1 
  ||v^{(p+1)}_{\ell}(\cdot\ t)||_{W^{n,\infty}(\Real)}
  (1-t)^p dt
  \leq 
  \frac{1}{p!}
\int_0^1 
\max(1,t^n)
  ||v^{(p+1)}_{\ell}||_{W^{n,\infty}(\Real)}
  (1-t)^p dt\\
  &\leq 
  \frac{1}{p!}
  ||v_{\ell}||_{W^{n+p+1,\infty}(\Real)},
\end{align*}
showing that also $V_{p,\ell}\in W^{n,\infty}(\Real)$ for all $n,p,\ell$.
Since \eq{etadeltasmall}
is satisfied
under the assumptions on $\eta$, $\delta$ and $k$,
we can use
Lemma~\ref{Imainpartairy} with  $n = \lceil(3\ell+4)/5\rceil$
to estimate
$$
  |I_{V_{p,\ell}}|\leq  C k^{\ell-(3\ell+3)/3}||V_{p,\ell}(\cdot/k^{\frac13})||_{W^{n,\infty}(\Real)}
  \leq C k^{-1}\max(1, k^{-\frac{n}3})||V_{p,\ell}||_{W^{n,\infty}(\Real)}\leq Ck^{-1}.
$$
For $I_Z$ we first observe that
\begin{align*}
z(\theta/k^{\frac13},k)=Z
 \left(\frac{k^{-\frac13} }{2}m_{11}(\xi_0+k^{-\frac13}\theta)\theta^4\right).
\end{align*}
By appealing to
Lemma~\ref{lem:Westimate} with $\varepsilon=k^{-1/3}$ 
we conclude that
$||z(\cdot/k^{\frac13},k)||_{W^{3,\infty}(\Real)}$
is bounded uniformly for $k\geq 1$.
Consequently, we can use Lemma~\ref{Imainpartairy} with $n=\lceil (12+1)/5\rceil =3$ together with
\eq{Wnineq1}
and \eq{Wnineq2} to show that
$$
 |I_{Z}|\leq  Ck^{3-12/3}||w(\cdot/k^{\frac13},k)||_{W^{3,\infty}(\Real)}
 \leq Ck^{-1}\max(1,k^{-1})||v_3||_{W^{3,\infty}(\Real)}
 ||z(\cdot/k^{\frac13},k)||_{W^{3,\infty}(\Real)}
 \leq C k^{-1}.
$$
We have thus proved that
\begin{equation}\label{ImainSest}
|I_{\rm main}-I_S|\leq C k^{-1}.
\end{equation}
From Lemma~\ref{IAiry} we know that
 $I_p\approx 2\pi
{\rm Ai}^{(p)}(\rho)i^p/k^{p/3}$
and we therefore introduce the approximation $\tilde{I}_S$ of $I_S$
obtained by replacing $I_p$ with the corresponding Airy function,
\begin{align*}
\tilde{I}_S &= 
v_0(0)2\pi{{\rm Ai}(\rho)}+
v_0'(0){\frac{2\pi i}{k^{1/3}}{\rm Ai}'(\rho)} +
\frac12v_0''(0)
{\frac{2\pi i^2}{k^{2/3}}{\rm Ai}^{(2)}(\rho)}
\\
&\ \ \ +v_1(0)\frac{ik}2
{\frac{2\pi i^4}{k^{4/3}}{\rm Ai}^{(4)}(\rho)}
 +\frac{ik}2 v_1'(0)
 {\frac{2\pi i^5}{k^{5/3}}{\rm Ai}^{(5)}(\rho)}
+\frac{(ik)^2}8 v_2(0){\frac{2\pi i^8}{k^{8/3}}{\rm Ai}^{(8)}(\rho)}.
\end{align*}
By \eq{Rbound} we can use Lemma~\ref{IAiry}
with large enough $n$ 
to obtain
$$
  \left|I_p-{\frac{2\pi i^p}{k^{p/3}}{\rm Ai}^{(p)}(\rho)}\right|
  \leq Ck^{-3},
$$
where $C$ is uniform in $\rho$. 
Then
\begin{equation}\label{ISStest}
  |I_S-\tilde{I}_S| \leq C(|v_0(0)|+|v_0'(0)|+
  |v_0''(0)|+k|v_1(0)|+k|v_1'(0)|+k^2|v_2(0)|)k^{-3}\leq C'k^{-1},
\end{equation}
as $||v_\ell||_{W^{n,\infty}(\Real)}\leq C$ uniformly in $\xi_0$.

The next step is to show that $\tilde{I}_S$ is close to ${\rm Ai}(\rho)$.
Upon using the identities for $\Ai^{(m)}$ with $m=2,4,5,8$ given
in Remark~\ref{polyexp} we can simplify the expression
for $\tilde{I}_S$ as follows
\begin{align*}
\frac{\tilde{I}_S}{2\pi} 
&= 
v_0(0){{\rm Ai}(\rho)}+
iv_0'(0)k^{-\frac13}{{\rm Ai}'(\rho)} 
-\frac12v_0''(0)k^{-\frac23}
{{\rm Ai}^{(2)}(\rho)}
\\
&\ \ \ +\frac{i}2v_1(0)
k^{-\frac13}{{\rm Ai}^{(4)}(\rho)}
 -\frac{1}2 v_1'(0)
 k^{-\frac23} {{\rm Ai}^{(5)}(\rho)}
-\frac{1}8 v_2(0)k^{-\frac23}{{\rm Ai}^{(8)}(\rho)}
\\
&=
v_0(0){{\rm Ai}(\rho)}+
iv_0'(0)k^{-\frac13}{{\rm Ai}'(\rho)} 
-\frac12v_0''(0)k^{-\frac23}
\rho{{\rm Ai}(\rho)}
+\frac{i}2v_1(0)
k^{-\frac13}
(\rho^2{{\rm Ai}(\rho)}+
2{{\rm Ai}'(\rho)})
\\
&\ \ \ 
 -\frac{1}2 v_1'(0)
 k^{-\frac23} 
(4\rho{{\rm Ai}(\rho)}
 +\rho^2{{\rm Ai}'(\rho)})
-\frac{1}8 v_2(0)k^{-\frac23}
((\rho^4+28\rho){{\rm Ai}(\rho)}+
12\rho^2{{\rm Ai}'(\rho)})
\\
&=
\left(v_0(0)
-\frac12v_0''(0)k^{-\frac23}\rho
+\frac{i}2v_1(0)k^{-\frac13}\rho^2
-2 v_1'(0)k^{-\frac23}\rho
-\frac{1}8 v_2(0)k^{-\frac23}(\rho^4+28\rho)
\right){{\rm Ai}(\rho)}
\\
&\ \ \ 
+\left(iv_0'(0)k^{-\frac13}
+iv_1(0)k^{-\frac13}
-\frac12 v_1'(0)k^{-\frac23} \rho^2
-\frac{3}2 v_2(0)k^{-\frac23}\rho^2\right)
{{\rm Ai}'(\rho)}.
\end{align*}
From \eq{rdef}, \eq{qdef} and \eq{m11}, we obtain
$$
  v_0(0) = \frac{1}{2\pi},\qquad v'_0(0)=-\frac{q'(\xi_0)}{4\pi q(\xi_0)},\qquad
  v_1(0)=\frac{2i-2\xi_0\beta}{4\pi q(\xi_0)}
  =\frac{q'(\xi_0)}{4\pi q(\xi_0)}.
$$
Hence, $v'_0(0)+v_1(0)=0$ and
\begin{align*}
 \tilde{I}_S &= 
\left(1
+{i\pi}k^{-\frac13}\rho^2v_1(0)
-\pi k^{-\frac23} \rho\left(v_0''(0)+
4v_1'(0)
+\frac{1}4 v_2(0)(\rho^3+28)\right)
\right){{\rm Ai}(\rho)}
\\
&\ \ \ 
-\pi k^{-\frac23} \rho^2 \left(
v_1'(0)
+3v_2(0)\right)
{{\rm Ai}'(\rho)}.
\end{align*}
Since $\rho$ is bounded by \eq{rhobound} and
$\Ai$ is smooth
around $\rho=0$,
this shows that
\begin{align}\label{IStAiest}
  |\tilde{I}_S-{{\rm Ai}(\rho)}|\leq 
  C\left[k^{-\frac13}\rho^2 +
  k^{-\frac23}|\rho|\right]
  \leq
  C'\left[k|\eta|^2 +
|\eta| \right]
  =
  C'\left[|k\eta|^2 +
|k\eta| \right]k^{-1}.
\end{align}
Note that the dependence on $k^2\eta^2$
which appears here, also appears in
the estimate of $R_2$ in Lemma~\ref{R2lemma}.

It remains to estimate $I_{t,\rm tail}$.
By \eq{Rbound} we get from
 Lemma~\ref{Itail} with $p=0$ and $n=2$, for all $t>R$, that
\begin{equation}\label{Itailest}
  |I_{t,\rm tail}|\leq 
  C_2k^{1/3-2}
  ||r||_{W^{2,\infty}(\Real)}
  \leq Ck^{-5/3},
\end{equation}
where the constant is independent of $t$.
In conclusion, 
using \eq{ImainSest}, \eq{ISStest}, \eq{IStAiest} and \eq{Itailest}
we have shown that
\begin{align*}
  |I(\eta,x_c,k) - {{\rm Ai}(\rho)}|  
  &\leq 
  |I_{\rm main} - I_S|+
   |I_S- \tilde{I}_S|
   +|\tilde{I}_S - {{\rm Ai}(\rho)}|
   +\lim_{t\to\infty} |I_{t,\rm tail}|
   \\
   &\leq
     C'\left[1+|k\eta|^2+|k\eta| + k^{-2/3}\right]k^{-1}
\leq
     C''\left[1+|k\eta|^2\right]k^{-1}.
\end{align*}
Together with \eqref{R3est}
this concludes the proof of Lemma~\ref{R3lemma}. 

Finally note that, away from the caustic point, i.e. $x<x_c$,
the method used here to estimate $R_3$ will not give sharp results;
if the stationary phase method is applied directly to $\phi_g$
extra decay in $k$ follows.
\end{proof}

\section{Properties of the Airy function}
\label{section:Airy}

Here we show some known properties of the Airy function and we
derive a few consequences in two lemmas. A more complete source
for information about Airy functions
is \cite{NIST:DLMF}, which we frequently cite below.
We consider the
 Airy function of the first kind $\Ai$ and second kind $\Bi$.
\begin{itemize}
\item[(P1)] The Airy functions are linearly independent solutions of
the Airy differential equation
\be\label{eq:Airydiff}
  \Ai''(z) = z\Ai(z),\qquad
  \Bi''(z) = z\Bi(z).
\ee
\item[(P2)] $\Ai$ and $\Ai'$ only have zeros on the negative real line. The
zeros do not coincide.
$\Ai(s)$ is positive and decreasing for $s\geq 0$.
\item[(P3)] $\Bi$ and $\Bi'$ 
also only have zeros on the negative real line. 
The
zeros do not coincide.
 $\Bi(s)$ is positive and increasing for $s\geq 0$.
\item[(P4)] Let
\be\label{Aiassympdef}
  \widetilde\Ai(z) := \frac12 \pi^{-\frac12}z^{-\frac14}e^{-\frac23 z^{\frac32}}.
\ee
Then, for real $s>0$,
$$
 \widetilde\Ai(-s) = \frac12
\pi^{-\frac12}s^{-\frac14}
\left(
\cos\left(\frac23 s^{\frac32}-\frac{\pi}{4}\right)+
i\sin\left(\frac23 s^{\frac32}-\frac{\pi}{4}\right)
\right),
$$
and it follows easily 
from \cite[Section 9.7 (ii,iii)]{NIST:DLMF} that
\be\label{eq:NIST1}
  \left|\Ai(s)-\widetilde\Ai(s)\right|
  \leq Cs^{-3/2}  \left|\widetilde\Ai(s)\right|,\qquad
  \left|\Ai(-s)-2\Re\widetilde\Ai(-s)\right|
  \leq Cs^{-3/2}
\ee
\be\label{eq:NIST2}
  \left|\Ai'(s)+\sqrt{s}\widetilde\Ai(s)\right|
  \leq Cs^{-1}  \left|\widetilde\Ai(s)\right|,\qquad
  \left|\Ai'(-s)-2\sqrt{s}\Im\widetilde\Ai(-s)\right|
  \leq Cs^{-1}.
\ee
 
\end{itemize}

We can now prove the following lemmas.

\begin{lem}\label{lem:Aiprop}
Let $\beta=e^{i\theta}$ with $|\theta|\leq \pi/3$.
There is a constant $C$ such that
\begin{align}
\left|\Ai(\beta s)-\widetilde\Ai(\beta s)\right| &\leq 
  Cs^{-3/2}\left|\widetilde\Ai(\beta s)\right|,\qquad s>0.\label{eq:Aidiff1}
\end{align}
Moreover, for each $s_0>0$ there is a constant $C(s_0)$ such that
\begin{align}
\left|\Ai(\beta s)-\widetilde\Ai(\beta s)\right| &\leq 
  C(s_0)s^{-3/2}\left|\Ai(\beta s)\right|, \qquad s\geq s_0.\label{eq:Aidiff2}
\end{align}
Moreover, for $s\in\Real$ and $\alpha = \exp(i\pi/3)$ there is a constant $C$ such that
\begin{align}
  |\Ai(s)| &\leq C (1+|s|)^{-1/4},\label{eq:Aisest}\\
  |\Ai'(s)| &\leq C (1+|s|)^{1/4},\label{eq:Aiprimesest}\\
  |\Ai(\alpha s)| &\geq C (1+|s|)^{-1/4},\label{eq:Aiasest}\\
|\Ai'(s) +i \sqrt{-s}\Ai(s)| &\geq 
C \begin{cases}
(1+|s|)^{1/2}|\Ai(s)|, & s\geq0,\\
(1+|s|)^{1/4}, & s<0.
\end{cases}\label{eq:Aicombest}
\end{align}
\end{lem}

\begin{proof}
We define $\Phi_0$ by the relation $\Ai(z)=\widetilde{\Ai}(z)\Phi_0(z)$.
It has an asymptotic expansion,
$$
  \Phi_0(z) \simeq \sum_{n=0}^\infty c_n z^{-3n/2},\qquad
  c_n = \frac{(-1)^n\Gamma(n+5/6)\Gamma(n+1/6)\left(\frac{3}{4}\right)^n}
  {2\pi n!},\qquad c_0 = 1, \qquad c_1\approx -0.104.
$$
From the estimates on $P$ and $Q$ of 
\cite[Appendix A, Lemma 7]{HakimOlivier:18} one obtains the uniform estimate
$$
\vert \Phi_0(z)-1\vert \leq \frac{\vert c_1\vert}{\vert z\vert^{3/2}},
$$
valid for $\vert\arg(z)\vert \leq\frac{\pi}{3}$. Then
(\ref{eq:Aidiff1}) follows directly with $C=|c_1|$,
$$
\left|\Ai(\beta s)-\widetilde\Ai(\beta s)\right|=
\left|\widetilde\Ai(\beta s)\right|\vert \Phi_0(\beta s)-1\vert
\leq |c_1|\frac{\left|\widetilde\Ai(\beta s)\right|}{s^{3/2}}.
$$
Now suppose $s\geq s_0>0$.
We first note that for all $s>0$,
$$
\vert \Phi_0(\beta s)\vert \geq 1-\vert \Phi_0(\beta s)-1\vert
\geq 1-
\frac{\vert c_1\vert}{s^{3/2}}.
$$
Hence, for $s\geq s_1:=(2|c_1|)^{2/3}$ we have $\vert \Phi_0(s\beta)\vert\geq 1/2$,
and since $\Ai=\widetilde{\Ai}\Phi_0$ only has zeros on the negative real line, there is a positive
infinum, 
$$
  d(s_0) := \inf_{s\geq s_0}\vert \Phi_0(\beta s)\vert \geq \min\left(\frac12,\ 
  \min_{s_0\leq s\leq s_1}\vert \Phi_0(\beta s)\vert\right) >0.
$$
Therefore, as above,
$$
\left|\Ai(\beta s)-\widetilde\Ai(\beta s)\right|
\leq 
|c_1|\frac{\left|\widetilde\Ai(\beta s)\right|}{s^{3/2}}.
\leq 
\frac{|c_1|}{d(s_0)}\frac{\left|\widetilde\Ai(\beta s)\Phi_0(\beta s)\right|}{s^{3/2}}
=
\frac{|c_1|}{d(s_0)}\frac{\left|\Ai(\beta s)\right|}{s^{3/2}},
$$
when $s\geq s_0$, proving (\ref{eq:Aidiff2}) with $C(s_0) = |c_1|/d(s_0)$.
 
Next, to show (\ref{eq:Aisest}) and (\ref{eq:Aiprimesest}) 
we note first that, for real $s>0$,
$$
  |\widetilde\Ai(s)| = 
  \left|\frac12 \pi^{-\frac12}s^{-\frac14}e^{-\frac23 s^{\frac32}}
  \right|
  \leq Cs^{-\frac14},\qquad
  |\widetilde\Ai(-s)| = 
  \left|\frac12 \pi^{-\frac12}(-s)^{-\frac14}e^{\frac23 is^{\frac32}}
  \right|
  \leq Cs^{-\frac14}.
$$
Then 
(\ref{eq:NIST1}, \ref{eq:NIST2}) readily give 
$$
  \left|\Ai(s)\right|\leq \left|\widetilde\Ai(s)\right|+
  C\left(1+\left|\widetilde\Ai(s)\right|\right)|s|^{-3/2}
  \leq C|s|^{-1/4}
$$
and
$$
  \left|\Ai'(s)\right|\leq \sqrt{|s|}\left|\widetilde\Ai(s)\right|+
  C\left(1+\left|\widetilde\Ai(s)\right|\right)|s|^{-1}
  \leq C|s|^{1/4}.
$$
which extend to  (\ref{eq:Aisest}) and (\ref{eq:Aiprimesest}) as $\Ai(0)$ is bounded.
 
The lower bound (\ref{eq:Aiasest})
follows for $s\geq s_0$ from the previous estimates,
$$
\left|\widetilde\Ai(\alpha s)\right|\leq
\left|\Ai(\alpha s)-\widetilde\Ai(\alpha s)\right|+
\left|\Ai(\alpha s)\right|
\leq
|c_1|
\frac{\left|\widetilde\Ai(\alpha s)\right|}{s^{3/2}}
+
\left|\Ai(\alpha s)\right|=
\frac12
\frac{s_0^{3/2}}{s^{3/2}}\left|\widetilde\Ai(\alpha s)\right|
+
\left|\Ai(\alpha s)\right|,
$$
as then
$$
\left|\Ai(\alpha s)\right|
\geq \frac12 \left|\widetilde\Ai(\alpha s)\right|
  =\frac{1}{4\sqrt{\pi}}
  \left|\alpha ^{-\frac14}s^{-\frac14}e^{-\frac23 is^{\frac32}}\right|
  =\frac{1}{4\sqrt{\pi}}s^{-\frac14}.
$$
By (P2) 
we also have $|\Ai(\alpha s)|\geq c$ for $0\leq s\leq s_0$ and some $c>0$.
Moreover, the identity \cite[Eq. 9.2.11]{NIST:DLMF} 
\begin{equation}\label{eq:ABident}
 \Ai(\alpha s) = \frac{\bar\alpha}{2}[\Ai(-s)+i\Bi(-s)],
\end{equation}
and (P3) implies that  $|\Ai(\alpha s)|\geq |\Bi(-s)|/2=\Bi(-s)/2\geq \Bi(0)/2>0$ for
$s\leq 0$.
This gives the
bound (\ref{eq:Aiasest}) also for $s\leq s_0$.

For \eq{eq:Aicombest} we consider $s\geq 0$ and use (\ref{eq:NIST1}, \ref{eq:NIST2})
\begin{align*}
2\sqrt{s}|\widetilde\Ai(-s)|
&\leq 
|2\sqrt{s}\widetilde\Ai(-s)-i\Ai'(-s) - \sqrt{s}\Ai(-s)|
+  |i\Ai'(-s) + \sqrt{s}\Ai(-s)|\\
&\leq
|2\sqrt{s}\Im\widetilde\Ai(-s)-\Ai'(-s)|
+ \sqrt{s}|2\Re\widetilde\Ai(-s)- \Ai(-s)|
+  |i\Ai'(-s) + \sqrt{s}\Ai(-s)|\\
&\leq
C s^{-1}
+  |i\Ai'(-s) + \sqrt{s}\Ai(-s)|
=C s^{-1}
+  |\Ai'(-s) + i\sqrt{s}\Ai(-s)|.
\end{align*}
Hence, 
$$
|\Ai'(-s) + i\sqrt{s}\Ai(-s)| \geq 
2\sqrt{s}|\widetilde\Ai(-s)|-Cs^{-1}
=
Cs^{1/4}-Cs^{-1}.
$$
Since the zeros of $\Ai'$ and $\Ai$ do not coincide (no double roots)
we get
$$
|\Ai'(-s) + i\sqrt{s}\Ai(-s)|=|\Ai'(-s)| + \sqrt{s}|\Ai(-s)|\neq 0, \qquad s>0,
$$
and the estimate \eq{eq:Aicombest} for $s<0$ follows.
Moreover, by (P2),
when $s\geq 0$,
\begin{align*} 
|\Ai'(s) + i\sqrt{-s}\Ai(s)|&=
|\Ai'(s) - \sqrt{s}\Ai(s)|=
\sqrt{s}\Ai(s)-\Ai'(s)=
\sqrt{s}|\Ai(s)|+|\Ai'(s)|\\
&\geq \max\left(\sqrt{s}|\Ai(s)|,\
|\Ai(s)|\min_{0\leq t\leq s}
\frac{|\Ai'(t)|}{|\Ai(t)|}\right)
\geq \frac12(\sqrt{s}+C)|\Ai(s)|,
\end{align*}
which gives \eq{eq:Aicombest} for $s\geq 0$. 
Here we also used the fact that
 $\lim_{s\to+\infty}\frac{|\Ai'(s)|}{|\Ai(s)|}=
\lim_{s\to+\infty}\frac{|\sqrt{s}\widetilde\Ai(s)|}{|\widetilde\Ai(s)|}=\infty$
by \eq{eq:NIST1} and \eq{eq:NIST2}. 
\end{proof}

\begin{lem}\label{lem:Airyderiv}
For the Airy function we have
$$
   \Ai^{(m)}(x)=p_m(x) \Ai(x) + q_m(x) \Ai'(x),
$$
where $p_m$ and $q_m$ are polynomials given by the recursions
\be\label{eq:recursion}
  p_{m+1}= p_m'+xq_m, \qquad
  q_{m+1}= p_m+q_m', \qquad p_0=1,\quad q_0=0.
\ee
The degree of their sum satisfies
$deg(p_m+q_m)=\lfloor m/2 \rfloor$  and, for $|x|<1$,
$$  
  |p_m(x)|+|q_m(x)|\leq \frac{\left(d_{m+1}!\right)^2}{1-|x|},\qquad
d_m =\left\lfloor\frac{m}{2}\right\rfloor.
$$
Furthermore,
\be\label{eq:Aizero}
   \Ai^{(3p+2)}(0)=0,\qquad p=0,1,\ldots
\ee
\end{lem}
\begin{proof}
Using the form of $\Ai^{(m)}(x)$ given and using \eq{eq:Airydiff} we note that
$$
\Ai^{(m+1)} = p_m' \Ai + q_m' \Ai'(x)+
p_m\Ai' + q_m\Ai''
= p_m' \Ai + q_m' \Ai'(x)+
p_m\Ai' + xq_m \Ai,
$$
where we used the Airy differential equation $\Ai''=x\Ai$.
This gives the recursion \eq{eq:recursion}.
The statement about the degree is easily checked for $m=0,1$.
Suppose it holds upto a general $m\geq 2$. Then
\begin{align*}
   {\rm deg}(p_{m+1}+q_{m+1})&=
   {\rm deg}(p'_m+q'_m+p_m+xq_m)=
   {\rm deg}(p_m+xq_m)\\
&=    {\rm deg}(p'_{m-1}+xq_{m-1}+xp_{m-1}+xq'_{m-1})=
   {\rm deg}(x(p_{m-1}+q_{m-1}))
   \\
   &=
   {\rm deg}(p_{m-1}+q_{m-1})+1.
\end{align*}
That ${\rm deg}(p_{m}+q_{m})=d_m$ follows by induction. Since the polynomials all
have positive coefficients, it also follows that
$d_m = \max({\rm deg}(p_{m}),{\rm deg}(q_{m}))$.
For a polynomial $p$, let $|p|_\infty$ denote its largest coefficient
in magnitude. Then $|xp|_\infty=|p|_\infty$ and
$|p'|_\infty\leq {\rm deg}(p)|p|_\infty$. Consequently,
\begin{align*}
|p_{m+1}|_\infty+|q_{m+1}|_\infty
&=|p'_m+xq_m|_\infty+|p_m+q_m'|_\infty\leq
|p'_m|_\infty+|xq_m|_\infty+|p_m|_\infty+|q_m'|_\infty\\
&\leq {\rm deg}(p_m)|p_m|_\infty+|q_m|_\infty+|p_m|_\infty+{\rm deg}(q_m)|q_m|_\infty
\\
&\leq 
(\max({\rm deg}(p_{m}),{\rm deg}(q_{m}))+1)(|p_m|_\infty+|q_m|_\infty)
=
(d_m+1)(|p_m|_\infty+|q_m|_\infty).
\end{align*}
Therefore, since $|p_0|_\infty+|q_0|_\infty=1$,
$$
|p_{m}|_\infty+|q_{m}|_\infty
\leq \Pi_{\ell=0}^{m-1}(d_\ell+1)
\leq \Pi_{\ell=2}^{m+1}d_{\ell}
\leq (d_{m+1}!)^2.
$$
Finally, we have for a polynomial $p$ of degree $d$,
and $|x|<1$,
$$
   \frac{|p(x)|}{|p|_\infty}\leq
1 + |x|+\cdots +|x|^d
\leq \frac{1}{1-|x|}.
$$
The last statement of the lemma 
is known for $p=0$. Suppose it holds for $p$ and use
the Airy differential equation $\Ai''=x\Ai$.
That gives $\Ai^{(3p+2)}= x\Ai^{(3p)}+(3p-1)\Ai^{(3(p-1)+2)}$,
which shows the claim.
\end{proof}

\begin{rem}\label{polyexp}
The first few polynomials $p_m$ and $q_m$ in the theorem are given by
\begin{align*}
 \Ai^{(2)}(x) &= x \Ai(x), \\
 \Ai^{(3)}(x) &= \Ai(x)+ x\Ai'(x), \\
 \Ai^{(4)}(x) &= x^2\Ai(x)+ 2\Ai'(x), \\
 \Ai^{(5)}(x) &= 4x \Ai(x)+ x^2\Ai'(x), \\
 \Ai^{(6)}(x) &= (x^3+4) \Ai(x)+ 6x\Ai'(x), \\
 \Ai^{(7)}(x) &= 9x^2 \Ai(x)+ (x^3+10)\Ai'(x), \\
 \Ai^{(8)}(x) &= (x^4+28x) \Ai(x)+ 12x^2\Ai'(x).
\end{align*}
\end{rem}



\newpage

\appendix
\section{Proof of Lemma \ref{lem:Fregularize}}

We show this for $k=1$ so that $\Fcal_k=\Fcal$.
The case with general $k$ follows from a simple rescaling.
Let $\phi\in {\mathcal S}$ 
be a test function and
$\langle\,\cdot\,,\,\cdot\,\rangle$ 
the duality pairing between ${\mathcal S}$
and ${\mathcal S}'$.
Then since $f\psi_t\in L^\infty\subset{\mathcal S}'$ and, by dominated convergence,
\begin{align*}
\langle\Fcal(f\psi_t),\phi\rangle&=
\langle f\psi_t,\Fcal(\phi)\rangle
=\int f(x)\psi_t(x)\Fcal(\phi)(x)dx\\
&\to\int f(x)\Fcal(\phi)(x)dx
=
\langle f,\Fcal(\phi)\rangle=
\langle\Fcal(f),\phi\rangle,
\end{align*}
where we used the facts that
$\Fcal(\phi)\in{\mathcal S}\subset L^1$,
$|\psi_t|\leq 1$ for all $t$ and $f\psi_t\to f$
pointwise.
This is true for all $\phi\in {\mathcal S}$ 
and therefore $\Fcal(f\psi_t)\to \Fcal(f)$
in ${\mathcal S}'$, proving the first statement.

That $\Fcal((f\ast g)\psi_t)
\to \Fcal(f)\Fcal(g)$
follows from
the first statement since $f\ast g\in L^{\infty}$ when $g\in {\mathcal S}$
and $\Fcal(f\ast g)
=\Fcal(f)\Fcal(g)$.
The last part of the second statement is true,
since $\Fcal(g)\phi\in{\mathcal S}$
and therefore the first part gives
\begin{align*}
\langle\Fcal(f\psi_t)
\Fcal(g),\phi\rangle=
\langle\Fcal(f\psi_t),
\Fcal(g)\phi\rangle \to 
\langle\Fcal(f),
\Fcal(g)\phi\rangle=
\langle\Fcal(f)
\Fcal(g),\phi\rangle.
\end{align*}
This shows the lemma.

\section{Proof of the non-stationary phase identities}

Below is a proof of identities used in
 the non-stationary phase lemma.
 The identities show how
 the rewritten integral
depends
on the derivatives of the phase function.
In order to do that we use the spaces of functions
defined in \eq{Wdef} and \eq{Udef}.

\begin{lem}\label{nonstatphase}
Suppose $D\subset\Real$ is a bounded open set and $K\subset D$ is compact. 
Let $a\in W^{n,1}(D)$ and $\phi\in C^{n+1}(D)$. If $\phi'\neq 0$ on $K$
and $\supp\, a\subset K$, then there exist functions
$u_{\ell,n}$ such that
$$
  \int_{D} a(y) e^{i\phi(y)/\varepsilon} dy = 
  (ik)^{-n}
\sum_{\ell=0}^n 
\int_{K} a^{(\ell)}(y)
  \frac{u_{\ell,n}(y)}{{\phi'(y)}^{n}} e^{ik\phi(y)}dy,\qquad
u_{\ell,n} \in {\mathcal U}_{n-\ell}(\phi).
$$
\end{lem}
\begin{proof}
Define the differential operator,
$$
  L[a] := \left( \frac{a}{\phi'}\right)'.
$$
Then, since $\supp\, u\subset K\subset D$ and $|\phi'|>0$ on $K$ integration by parts gives
\begin{equation}\label{eq:inparest}
  \int_{D} a(y) e^{i\phi(y)/\varepsilon} dy
  =  \frac{1}{ik}\int_{D} L[a](y) e^{ik\phi(y)} dy.
\end{equation}
Since
 and $u$ and $\phi$ are sufficiently regular,
this can be repeated $n$ times, giving
$$
\int_{D} u(y) e^{ik\phi(y)} dy
  = 
  (ik)^{-n}\int_{D} L^n[a](y) e^{ik\phi(y)} dy.
$$
We thus need to show that there exist $u_{\ell,n}$ such that
$$
L^n[a]=
\sum_{\ell=0}^n 
a^{(\ell)}
  \frac{u_{\ell,n}}{{\phi'}^{n}},\qquad
  u_{\ell,n} \in {\mathcal U}_{n-\ell}(\phi).
$$
When $n=0$ this simply says that $u_{0,0}=1\in {\mathcal U}_{0}(\phi)$. 
Suppose the claim holds for $n$ and consider
\begin{align*}
L^{n+1}[a]&=
L
\sum_{\ell=0}^n 
a^{(\ell)}
  \frac{u_{\ell,n}}{{\phi'}^{n}}
=
\sum_{\ell=0}^n 
\frac{d}{dy}\left(a^{(\ell)}
  \frac{u_{\ell,n}}{{\phi'}^{n+1}}\right)
\\
&=
\sum_{\ell=0}^n 
a^{(\ell+1)}
  \frac{u_{\ell,n}}{{\phi'}^{n+1}}
  +
a^{(\ell)}
  \frac{u'_{\ell,n}}{{\phi'}^{n+1}}
  -(n+1)
a^{(\ell)}
  \frac{u_{\ell,n}\frac{\phi''}{\phi'}}{{\phi'}^{n+1}}.
\end{align*}
For the first term we have
$$
u_{\ell,n}\in {\mathcal U}_{n-\ell}(\phi)= {\mathcal U}_{n+1-(\ell+1)}(\phi).
$$
For the third term
$$
u_{\ell,n}\frac{\phi''}{\phi'}\in {\mathcal U}_{n+1-\ell}(\phi).
$$
For the second term, consider one basis function $w\in{\mathcal W}_p(\phi)$,
$$
   w =\prod_{k=0}^M\frac{\phi^{(\alpha_k+1)}}{\phi'},\qquad
   \sum_{k=0}^M\alpha_k=p.
$$
for some $M$. Then
$$
w'
=\sum_{\ell=0}^M\prod_{k=0}^M
\frac{\phi^{(\alpha_k+1+\delta_{\ell,k})}}{\phi'}
-\sum_{\ell=0}^M\frac{\phi''}{\phi'}\prod_{k=0}^M\frac{\phi^{(\alpha_k+1)}}{\phi'}
\in {\mathcal U}_{p+1}(\phi).
$$
Hence, $u'_{\ell,n}\in {\mathcal U}_{n+1-\ell}(\phi)$,
as it is a linear combination of derivatives of functions
in ${\mathcal U}_{n-\ell}(\phi)$.
This shows that $L^{n+1}[a]$ is of the correct form and the lemma
is proved.
\end{proof}

\section{Boundedness of $Z$}

Here we consider the scaled remainder term in the Taylor expansion
of $\exp(iz)$,
\begin{equation}\label{WdefA}
 Z(z) = \frac{1}{z^3}\left(
  e^{iz} - (1+iz +(iz)^2/2)
 \right) =
 \frac{1}{2i}\int_0^1 e^{isz}(1-s)^2ds.
\end{equation}
We have the following lemma.
\begin{lem}\label{lem:Westimate}
Let
$$
  \sigma^\varepsilon(\theta)=Z
 \left(\frac{\varepsilon }{2}m_{11}(\xi_0+\varepsilon\theta)\theta^4\right).
$$
Then there is a constant $C$ such that
$$
   ||\sigma^\varepsilon||_{W^{3,\infty}(\Real)}\leq C,
$$
for all $0<\varepsilon\leq 1$.
\end{lem}
\begin{proof}
We begin by estimating $Z$ and its first three derivatives
for $z$ with non-negative imaginary part.
Then all derivatives of $Z$ are bounded, since
$$
 \left|\frac{d^pZ(z)}{dz^p}\right| = \frac{1}2\left|\int_0^1 e^{isz}(1-s)^2(is)^pds\right|
 \leq 
 \frac{1}2\int_0^1 e^{-s\Im z}(1-s)^2 s^pds <\frac12.
$$
Furthermore, from the first part of the definition (\ref{WdefA}) 
we get, for $p=0,\ldots,3$ and $|z|\geq 1$,
$$
 \left|\frac{d^pZ(z)}{dz^p}\right|
 \leq C\left(\sum_{j=0}^p |z|^{-3-p+j} + |z|^{-3-p}+ |z|^{-2-p}+ |z|^{-1-p}
 \right)
 \leq \frac{C'}{|z|^{\min(3,p+1)}}.
$$
It follows that there is a constant $C$ such that
\begin{equation}\label{Westimate}
 \left|\frac{d^pZ(z)}{dt^p}\right|\leq \frac{C}{1+|z|^p},\qquad p=0,\ldots, 3,
 \qquad \Im z \geq 0.
\end{equation}
Next, let
$$
  G^\varepsilon(\theta) =\frac{\varepsilon }{2}m_{11}(\xi_0+\varepsilon\theta)\theta^4,
$$
so that $\sigma^\varepsilon(\theta) =Z(G^\varepsilon(\theta))$.
Then, by Lemma~\ref{m11est} for $0\leq p\leq 4$,
$$
  \left|\frac{d^pG^\varepsilon(\theta)}{d\theta^p}\right| 
  \leq C\varepsilon\sum_{j=0}^p \varepsilon^{j} |m_{11}^{(j)}(\xi_0+\varepsilon\theta)|
  |\theta|^{4-p+j}
  \leq 
  C\varepsilon |\theta|^{4-p} \sum_{j=0}^p 
  \frac{d_j|\varepsilon\theta|^{j}}
{1+|\varepsilon\theta|^{j+1}}  
  \leq C'\frac{\varepsilon|\theta|^{4-p}}
{1+|\varepsilon\theta|},
$$
and since $\Im m_{11}>0$ by Lemma~\ref{m11est}, we
get from \eq{Westimate} 
and
\eq{m11d} that
$$
  |Z^{(p)}(G^\varepsilon(\theta))|
  \leq \frac{C}{1+\left|G^\varepsilon(\theta)\right|^p}
  \leq \frac{C}{
  1+\left|\frac{D_0\varepsilon\theta^{4}}{2(1+|\varepsilon\theta|)}\right|^p}
  \leq C''\frac{(1+|\varepsilon\theta|)^p}{
  1+\varepsilon^p|\theta|^{4p}}, \qquad p=0,\ldots, 3.
$$
Therefore, 
if $p=0,\ldots,3$ and $j_1+\cdots+j_p=p'\leq 4p$,
\begin{align*}
\left|Z^{(p)}(G^\varepsilon(\theta))
\frac{d^{j_1}G^\varepsilon(\theta)}{d\theta^{j_1}}
\cdots \frac{d^{j_p}G^\varepsilon(\theta)}{d\theta^{j_p}}
\right|
&\leq 
\frac{C''(1+|\varepsilon\theta|)^p}
{
  1+\varepsilon^p|\theta|^{4p}}\frac{(C'\varepsilon)^{j_1+\cdots +j_p}|\theta|^{4-j_1+4-j_2+\cdots+4-j_p}}
{(1+|\varepsilon\theta|)^{j_1+\cdots +j_p}}
\\
&= C'''
\frac{\varepsilon^{p}
|\theta|^{4p-p'}
}
{1+
\varepsilon^p\left|
\theta
\right|^{4p}
}
= C'''
\frac{\varepsilon^{p'/4}
(\varepsilon^{1/4}|\theta|)^{4p-p'}
}
{1+
(\varepsilon^{1/4}|\theta|)^{4p}
}\leq C'''\varepsilon^{p'/4}.
\end{align*}
From these estimates we get, with $p=p'=0,\ldots, 3$,
\begin{align*}
\left| \sigma^\varepsilon(\theta)
\right|
&= |Z(G^\varepsilon(\theta))|\leq
C,
   \\
\left| \frac{d}{d\theta}\sigma^\varepsilon(\theta)
\right|
&= 
\left|Z^{(1)}(G^\varepsilon(\theta))
\frac{dG^\varepsilon(\theta)}{d\theta}
\right|
\leq
C\varepsilon^{1/4},
   \\
\left| \frac{d^2}{d\theta^2}\sigma^\varepsilon(\theta)
\right|
&=
\left|Z^{(2)}(G^\varepsilon(\theta))
\left(\frac{dG^\varepsilon(\theta)}{d\theta}\right)^2
+Z^{(1)}(G^\varepsilon(\theta))
\frac{d^{2}G^\varepsilon(\theta)}{d\theta^{2}}
\right|
\leq C\varepsilon^{1/2},
\\
\left| \frac{d^3}{d\theta^3}\sigma^\varepsilon(\theta)
\right|
&=
\left|Z^{(3)}(G^\varepsilon(\theta))
\left(\frac{dG^\varepsilon(\theta)}{d\theta}\right)^3
+3Z^{(2)}(G^\varepsilon(\theta))
\frac{d^{2}G^\varepsilon(\theta)}{d\theta^{2}}\frac{dG^\varepsilon(\theta)}{d\theta}
+3Z^{(1)}(G^\varepsilon(\theta))
\frac{d^{3}G^\varepsilon(\theta)}{d\theta^{3}}
\right|
\leq C\varepsilon^{3/4}.
\end{align*}
This shows the lemma.
\end{proof}

\begin{thebibliography}{99}

\bibitem{HakimOlivier:18} H. Boumaza and O. Lafitte.
\newblock The band spectrum of the periodic Airy--Schr\"odinger
operator on the real line.
\newblock {\em J. Differential Equations}, 264:455--505, 2018.

\bibitem{Liu3}
H.~Liu, J.~Ralston, O.~Runborg, and N.~M. Tanushev.
\newblock Gaussian beam method for the {H}elmholtz equation.
\newblock {\em SIAM J. Appl. Math.}, 74(3):771--793, 2014.

\bibitem{LiuRalston:20}
H.~Liu, J.~Ralston, and P. Yin.
\newblock General superpositions of Gaussian beams
and propagation errors.
\newblock {\em Math. Comp.}, 89:675--697, 2020.

\bibitem{Liu1}
H.~Liu, O.~Runborg, and N.~M. Tanushev.
\newblock Error estimates for {G}aussian beam superpositions.
\newblock {\em Math. Comp.}, 82:919--952, 2013.

\bibitem{Liu2}
H. Liu, O. Runborg and N. Tanushev. Sobolev and max norm error estimates for Gaussian beam superpositions. {\it Commun. Math. Sci.}, 
14(7):2037-2072, 2016.




\bibitem{Ludwig} D.~Ludwig.
\newblock Uniform asymptotic expansions at a caustic.
\newblock {\em Commun. Pur. Appl. Math.}, 19:215--250, 1966.



%
%

\bibitem{NIST:DLMF}
{\it NIST Digital Library of Mathematical Functions}.
\newblock http://dlmf.nist.gov/, Release 1.1.0 of 2020-12-15.
\newblock F.~W.~J. Olver, A.~B. {Olde Daalhuis}, D.~W. Lozier, B.~I. Schneider,
  R.~F. Boisvert, C.~W. Clark, B.~R. Miller, B.~V. Saunders, H.~S. Cohl, and
  M.~A. McClain, eds.

\bibitem{Ralston:82}
J.~Ralston.
\newblock Gaussian beams and the propagation of singularities.
\newblock In {\em Studies in partial differential equations}, volume~23 of {\em
  MAA Stud. Math.}, pages 206--248. Math. Assoc. America, Washington, DC, 1982.

\bibitem{Zheng:14}
C.~Zheng.
\newblock Optimal error estimates for first-order {G}aussian beam
  approximations to the {S}chr\"odinger equation.
\newblock {\em SIAM J. Num. Anal.}, 52(6):2905--2930, 2014.

\end{thebibliography}
\end{document}